\documentclass[b5paper] {article}
[12pt]

 \usepackage{booktabs}
\usepackage[square,sort&compress,numbers]{natbib}

\makeatletter
\renewcommand*\l@section{\@dottedtocline{1}{1.5em}{2.3em}}
\makeatother

\usepackage{amsfonts}
\usepackage{amssymb}
\usepackage[T1]{fontenc}

\usepackage{tikz}
\usetikzlibrary{calc}

\usepackage{CJK}
\usepackage{amsmath}

\usepackage{amsfonts}
\usepackage{amssymb}
\usepackage{amsthm}
\usepackage{amssymb}
\usepackage{enumerate}
\usepackage[calc]{picture}
\usepackage[all,cmtip]{xy}

\usepackage[mathscr]{eucal}
\usepackage{eqlist}

\usepackage{color}
\usepackage{abstract}
\usepackage[T1]{fontenc}

\usepackage[top=2.26cm, bottom=2.26cm, left=2.0cm, right= 2.0cm]{geometry}

\theoremstyle{plain}
\newtheorem{theorem}{Theorem}
\newtheorem{proposition}[theorem]{Proposition}
\newtheorem{lemma}[theorem]{Lemma}

\newtheorem{corollary}[theorem]{Corollary}

\theoremstyle{definition}
\newtheorem{example}{Example}

\theoremstyle{definition}
\newtheorem{definition}{Definition}

\usepackage{etoolbox}
\newtheoremstyle{myrem}
 {3pt}
 {3pt}
 {\normalsize}
 { }
 {\itshape}
 {:}
 { }
 {}

 \theoremstyle{myrem}
 \newtheorem{remark}{Remark}
 \appto\remark{\leftskip\parindent}
 \appto\remark{\rightskip\parindent}

\usepackage{float}

\numberwithin{equation}{section}
\numberwithin{theorem}{section}

\begin{document}

~~~
\vspace{1cm}

\begin{center}
{\Large {\textbf {The  stability  of  persistent homology of  hypergraphs }}}
 \vspace{0.38cm}\\

{\large   Shiquan Ren,   Jie Wu*  }

\footnotetext[1] { *  corresponding  author.   }

\footnotetext[2]
{{\bf Financial Support. } Shiquan Ren was supported by  China Postdoctoral Science Foundation  (Grant no.  2020M680494, 2022M721023)   and Natural Science Foundation of China (NSFC  grant no. 12001310).    }

\footnotetext[3]
{{\bf Financial Support. }
Jie Wu was supported by Natural Science Foundation of China (NSFC grant no. 11971144),   High-level Scientific Research Foundation of Hebei Province,  and  the  start-up  research  fund  from  BIMSA.    }

\footnotetext[4]
{{\bf Acknowledgement. } The  authors  would like to express  their  deep
gratitude to the  editor  and the referee for  their  kind  helps.  }

\begin{quote}
\begin{abstract}
 Hypergraph  is  the most general  model   for complex networks involving group interactions.   Taking   the  ideas    of  path  homology     from  Alexander Grigor'yan,  Yong  Lin, Yuri Muranov  and  Shing-Tung Yau  \cite{lin1,lin2,lin3,lin4,lin6},   Stephane  Bressan, Jingyan   Li and the authors of this article   introduced   embedded homology  of hypergraphs~\cite{h1} in  2019,   which     has leaded to successful applications in protein-ligand binding network~\cite{liu-xia, lx1} in 2021.   A fundamental question  arising from practical applications  is about the stability of the persistent  embedded homology of hypergraphs.
   In this paper,  we prove  the  stability of the persistent  embedded homology as  well  as  the  persistent  homology of  the   associated simplicial complex     with respect to perturbations of the filtration  on  a  hypergraph.   We      apply  the  persistent homology  methods  to      morphisms  of  hypergraphs  and
   prove  the  stability  with respect  to  perturbations  of  the  filtrations.     We   prove  the  constancy   of  the persistent  Betti numbers   under some conditions on  the  simple-homotopy  types of  hypergraphs.
    \end{abstract}
\end{quote}

\end{center}


\noindent
{{\bf Key  words and Phrases.}   hypergraphs,  simplicial complexes,   persistent homology, stability  }

\section{Introduction}

 Let $V$ be a finite set equipped with a total order.
  Let $2^V$ denote the power-set of $V$.
  Let $\emptyset$ denote the empty set.  A {\it hypergraph}    $\mathcal{H}$ on $V$ is a subset of  $2^V\setminus\{\emptyset\}$ (cf. \cite{berge, parks}).  We call $V$ the {\it vertex-set} and call an element of $V$  a {\it vertex}.   We call an element in $\mathcal{H}$  a  {\it  hyperedge}. From the view of algebraic topology, a hypergraph may be considered as an \textit{incomplete} simplicial complex with missing some faces. More precisely,
  if   a  hypergraph satisfies  the   condition:
  {\it
 for any $\sigma\in \mathcal{H}$ and any  nonempty  subset $\tau\subseteq \sigma$,    the set $\tau$ must be a hyperedge in  $\mathcal{H}$},
 then  it  is an (abstract)  simplicial complex.
 A hyperedge  of a simplicial complex is   a   simplex.   Let $\mathcal{H}$  be a hypergraph on $V$ and let $\mathcal{H}'$  be a hypergraph on  $V'$.
A  morphism  $\varphi: \mathcal{H}\longrightarrow \mathcal{H}'$   is
 a  map  $\varphi:  V\longrightarrow V'$  such that  each  hyperedge of $\mathcal{H}$  is  sent to a hyperedge of $\mathcal{H}'$.   A   morphism of hypergraphs between two  simplicial complexes  is     a     simplicial map.

 People have been  used  networks having group interactions    in  various  disciplines.   Pairwise representations   are not necessarily appropriate when the fundamental interactions involve more than two entities at the same time, such as bionetworks on protein-protein interactions, protein-ligand interactions and microbiome interactions as well as many social networks.   Hypergraph  is  the most general mathematical model without constraint conditions for complex networks involving group interactions, which has been extensively used by the recent researchers  \cite{interaction3, interaction4,  interaction5}, see the review articles \cite{physr,rev111,review}. This demands mathematical explorations on hypergraphs for introducing (more) mathematical features that can detect the structures of hypergraphs.

Since 2010s,   Alexander Grigor'yan,  Yong  Lin, Yuri Muranov  and  Shing-Tung Yau  started to  study    the    topological features   of   graphs  and  digraphs, with introducing path homology theory for graphs and digraphs in their series work~\cite{lin1,lin2,lin3,lin4,lin6}, which breaks through some frames of the classical algebraic topology in the sense that their homology theory is able to apply to \textit{incomplete complexes with missing faces}. Taking the ideas of path homology, Stephane  Bressan, Jingyan   Li and the authors of this article introduced \textit{embedded homology} of hypergraphs~\cite{h1} in 2019, which seems one of the most canonical extensions of simplicial homology to hypergraphs in the sense that the representation of a hypergraph in a chain complex induces infimum complex and supremum complex having the same chain homotopy type that is used for defining the embedded homology, see Subsection~\ref{hypergraph-homology} for details.

The newly introduced embedded homology of hypergraphs has leaded to successful applications in protein-ligand binding network~\cite{liu-xia, lx1} in 2021, where the persistent method for   hypergraphs plays an important role for the  experimental results. A fundamental question arising from practical applications is then whether the barcodes derived from persistent embedded homology is robust. In terms of mathematics, the question is about the stability of the persistent embedded homology of hypergraphs. The purpose of this article is to give the stability theorem on persistent embedded homology of hypergraphs.

Our methodology for studying the stability  of  the persistent embedded homology  for hypergraphs is similar to that on the stability of persistent simplicial homology  \cite{pmd,pd1},  while some technical difficulties have to be overcome in the case of hypergraphs. Since the embedded homology of hypergraphs is a canonical extension of simplicial homology, our stability theorem on  persistent embedded homology of hypergraphs  can be viewed as a generalization of the stability theorem on persistent simplicial homology that gives a robust topological approach to data modeled by hypergraph.

\smallskip

Now let us give the detailed discussions. We consider real-valued functions on hypergraphs and use the persistent of the embedded homology groups  as  well as  the  homology  of  the associated simplicial complexes  to study their qualitative and quantitative behavior.  Specifically,  we study the  topological features of hypergraphs as well as morphisms  between hypergraphs.  By applying the persistent  of  the embedded  homology as  well as  the  homology  of  the associated simplicial complexes  to the induced pull-back filtrations as well as the induced push-forward filtrations, we use persistent diagrams to encode the topological features of the morphisms and prove the stability of this encoding.
The stability that will be  proved in this paper    gives the theoretical foundation    for the applications of the persistent   embedded homology    for hypergraphs,  which says that the out-put of the persistent homology is stable under small perturbations   of the observation quantity.

The first main result proves that  the persistent diagram   of the persistent homology for a hypergraph  is  stable   with respect to  perturbations of a filtration of the hypergraph.
Let  $\mathcal{H}$  be  a  hypergraph. Let  $f$  be a  real function on $\mathcal{H}$.   Let  $\mathcal{H}^f$  be  the filtration
  of  $\mathcal{H}$  with respect to  $f$,  i.e. $\mathcal{H}^f$  is  a family of  hypergraphs $\mathcal{H}_t^f$,  $t\in\mathcal{R}$,  where $\mathcal{H}^f_t=f^{-1}(-\infty,t]$.
  We  have  the persistent homology $H_n(\Delta(\mathcal{H}^f_t))$   of the filtered associated simplicial complex  $\Delta(\mathcal{H}^f_t)$  whose  persistent  diagram is denoted  by  $D(\Delta\mathcal{H}, f)$,  the persistent homology $H_n(\delta(\mathcal{H}^f_t))$  of the lower-associated  simplicial complex $\delta(\mathcal{H}^f_t)$  whose  persistent  diagram is denoted  by  $D(\delta\mathcal{H}, f)$,      and  the  persistent of the embedded homology $H_n(\mathcal{H}^f_t)$  whose  persistent  diagram is denoted  by  $D(\mathcal{H}, f)$.
   Let  $ D\mathcal{H}^f$  be  the
  triple $(D(\Delta\mathcal{H}, f), D(\delta\mathcal{H}, f),  D(\mathcal{H}, f))$.
   Then the  stability  of the  persistent diagram  makes  sense  for   $ D\mathcal{H}^f$.

  \begin{theorem}[Main Result I]
 \label{th-main111}
 Let $\mathcal{H}$ be a hypergraph. Let $f$ and $g$  be real-valued functions on $\mathcal{H}$. Then
the bottleneck distance $d_B^\infty(D\mathcal{H}^f, D\mathcal{H}^g)$ between the persistent diagrams induced from $\mathcal{H}^f$ and $\mathcal{H}^g$ is bounded above  by  the $L^\infty$-distance $||f-g||_{\infty}$ between $f$ and $g$.
 \end{theorem}

   The second main result proves  that   the persistent diagram of a morphism between two hypergraphs  is  stable    with respect to perturbations of the  filtrations  on both the source  hypergraph and the target   hypergraph.
     Let  $\varphi:  \mathcal{H}\longrightarrow \mathcal{H}'$  be a  morphism  of hypergraphs.     Let  $f$   be  a  real  function  on  $\mathcal{H}$  and  let  $f'$  be  a  real  function  on  $\mathcal{H}'$.   Then  we  have a  filtration  $\mathcal{H}_t^f$  of  $\mathcal{H}$  and  a  filtration  ${\mathcal{H}'}_t^{f'}$    of    $\mathcal{H}'$.
 We  will define   in  Subsection~\ref{s666} an  induced filtration of   $\mathcal{H}$  by pulling  back  ${\mathcal{H}'}_t^{f'}$    via  $\varphi$
    and  define  twenty-one  induced homomorphisms of  persistent homology groups,   denoted as  $\Phi _{f }$,     in the  commutative  diagram  in   Subsection~\ref{s666}~(A).
      We  will  also   define   in  Subsection~\ref{s666} an  induced filtration of $\mathcal{H}'$  by pushing  forward   ${\mathcal{H}}_t^{f}$    via  $\varphi$ and  define   twenty-one  induced  homomorphisms  of   persistent  homology  groups,  denoted as  $\Phi'_{f'}$,   in the  commutative  diagram  in   Subsection~\ref{s666}~(B).   We  define the  persistent diagram of  a  homomorphism between two persistent homology groups to be the triple consisting  of  the  persistent diagram of  the kernel,   the  persistent diagram of   the  image, and   the  persistent diagram of     the  cokernel.
  Then the  stability  of the  persistent diagram  makes  sense  for   $\Phi_{f}$ and  $\Phi'_{f'}$.

 \begin{theorem}[Main Result II]\label{main2}
Let $\epsilon>0$.  Suppose the  persistent  homology is with field coefficients.  \begin{enumerate}[(i).]
\item
If $||f-g||_\infty\leq \epsilon$,  then each of the  persistent  linear map among the twenty-one   persistent  linear maps in the commutative diagram in (B) (cf. Subsection~\ref{s666}), denoted as $\Phi_f$ and $\Phi_g$, satisfies
  $d_B^\infty(D(\Phi_f), D(\Phi_g))\leq \epsilon$;

\item
If $||f'-g'||_\infty\leq \epsilon$,   then each of the  persistent  linear map among the twenty-one    persistent  linear maps in the commutative diagram in (A) (cf. Subsection~\ref{s666}), denoted as $\Phi'_{f'}$ and $\Phi'_{g'}$, satisfies
  $ d_B^\infty(D(\Phi'_{f'}), D(\Phi'_{g'}))\leq \epsilon$.
\end{enumerate}
 \end{theorem}

\smallskip

Thanks  to the  stability theorem of the persistent homology of simplicial complexes,    Theorem~\ref{th-main111}  and   Theorem~\ref{main2} can  be  proved.
The notion of persistent diagram was  introduced in \cite{2002}.
  The stability
  of persistent diagrams   was  proved in  \cite{pd1}.   An interleaving condition was  discovered in  \cite{pd2},  which provides a way to measure the distance between persistent modules and ensure the stability.      Homomorphisms of persistent modules   and the stability  was  studied in    \cite{add-1}.    The   methodology   for   the proofs of Theorem~\ref{th-main111}  and   Theorem~\ref{main2}  is  originated  from    \cite{pd2,pmd,pd1,add-1,decomp,2010ams,2002} with some technical details overcome.

   For   any  real  numbers  $t\leq  s$,   the  persistent  Betti number   $\beta_n^{t,s}$  is  the  dimension  of  the  persistent  homology  $H_n^{t,s}$  (cf.  \cite{2005}).   We  define  the  simple-homotopy types  of  hypergraphs     in  Subsection~\ref{ssssssss3.11111111}.  For  any  hypergraph  $\mathcal{H}$  and  any  real  function   $f$  on  $\mathcal{H}$,   the  third main result  proves  the  constancy   of  the   persistent  Betti  numbers  under  some   conditions on the   simple-homotopy  types  of  the  filtered   hypergraphs.

  \begin{theorem}[Main Result  III]
    Let  $\mathcal{H}$  be  a  hypergraph  and  let  $f:  \mathcal{H}\longrightarrow \mathbb{R}$.         For  any   $t\leq  s$,
    \begin{enumerate}[(1).]
    \item
if    $\mathcal{H}_t^f$   and   $\mathcal{H}_s^f$
 have all the vertices as $0$-hyperedges and  have the same   strong  simple-homotopy type,    then the  persistent  Betti  number
 $\beta_n^{t,s}$    of  the embedded homology     is  constant;
  \item
  if  $\mathcal{H}_t^f$   and   $\mathcal{H}_s^f$    have the same  weak simple-homotopy type,      then
  the  persistent  Betti  number
  $\beta_{\Delta,n}^{t,s}  $  of  the  homology of  the  associated  simplicial complexes  is  constant;
    \item
    if   $\mathcal{H}_t^f$   and   $\mathcal{H}_s^f$ have the same  weak simple-homotopy type via a sequence of weak simplicial collapses and  weak simplicial expansions   such  that   in each  weak elementary  simplicial collapse or weak elementary  simplicial expansion,  $\eta$  is a free face of $\xi$ and one of (a), (b)  and (c) in Lemma~\ref{pr-2.8a182}  is satisfied,       then  the  persistent  Betti  number
  $\beta_{\delta,n}^{t,s}  $   of  the  homology of  the  lower-associated  simplicial complexes  is  constant.
    \end{enumerate}
    \end{theorem}

\section{The  homology of   hypergraphs}\label{s222}

  \subsection{Hypergraphs and  simplicial  complexes}

Let $V$  be a finite set with a total order $\prec$.
A  {\it hyperedge}    on $V$  is
  $\sigma=\{v_0,v_1,\ldots,v_n\}$
 where $v_0,v_1,\ldots, v_n\in V$  and  $v_0\prec v_1\prec\cdots \prec v_n$.  We  call $n$  the  {\it dimension} of $\sigma$ and call   $\sigma$   an {\it $n$-hyperedge}.   A {\it hypergraph} $\mathcal{H}$  on $V$   is a collection of hyperedges on $V$.
An {\it (abstract)  simplicial complex} $\mathcal{K}$  on $V$  is a hypergraph  on $V$  such that   $\tau\in\mathcal{K}$   for any $\sigma\in \mathcal{K}$  and any  nonempty  subset $\tau\subseteq \sigma$.    A  hyperedge of a simplicial complex  is  called a {\it simplex}.

 For a  hyperedge $\sigma=\{v_0,v_1,\ldots,v_n\}$ on $V$, the {\it associated simplicial complex} $\Delta\sigma$ of $\sigma$ is   the collection of all the nonempty subsets of $\sigma$
\begin{eqnarray*}\label{eq-md1}
\Delta\sigma=\{\{v_{i_0},v_{i_1},\ldots,v_{i_k}\}\mid 0\leq i_0<i_1<\cdots< i_k\leq n,  0\leq k\leq n\}.
\end{eqnarray*}
Let $\sigma=\{v_0,v_1,\ldots,v_n\}$   and $\tau=\{u_0,u_1,\ldots,u_m\}$  be two   hyperedges  on $V$.   Suppose $\sigma\cap\tau=\emptyset    $.  The {\it join}  $\sigma*\tau$   is  an $(n+m+1)$-hyperedge
\begin{eqnarray*}
\sigma*\tau =\{v_0,v_1,\ldots,v_n,u_0,u_1,\ldots,u_m\},
\end{eqnarray*}
where the right-hand side is up  to a rearrangement of the vertices $v_0$, $v_1$, $\ldots$, $v_n$, $u_0$, $u_1$, $\ldots$, $u_m$ with respect to the total order $\prec$.     It  is  direct  that
$\Delta(\sigma  *\tau)=(\Delta  \sigma)  *  (\Delta  \tau)$.
Let  $\mathcal{H}$  be  a hypergraph on $V$.     The  {\it associated simplicial complex}
    \begin{eqnarray*}\label{e10}
\Delta \mathcal{H}=\{\tau\in\Delta\sigma\mid \sigma\in \mathcal{H}
\}
\end{eqnarray*}
  is the smallest simplicial complex that $\mathcal{H}$ can be embedded in (cf. \cite{parks})  and     the {\it lower-associated simplicial complex}
  \begin{eqnarray*}\label{e11}
\delta\mathcal{H} = \{\sigma\in\mathcal{H}\mid \Delta\sigma\subseteq \mathcal{H}\}
  = \{\tau\in\Delta\sigma\mid \Delta\sigma\subseteq \mathcal{H}\}
\end{eqnarray*}
is the largest simplicial complex that can be embedded in $\mathcal{H}$.

Let $V$ and $V'$ be two totally-ordered finite sets.  Let $\mathcal{H}$ and $\mathcal{H}'$ be hypergraphs on $V$ and $V'$ respectively.  A {\it morphism} of hypergraphs from $\mathcal{H}$ to $\mathcal{H}'$ is  a map $\varphi$ from $V$ to $V'$
such that for any $k\geq 0$,  whenever $\sigma=\{v_0,\ldots,v_k\}$ is a hyperedge of $\mathcal{H}$,   its  image  $\varphi(\sigma) = \{\varphi(v_0), \ldots, \varphi(v_k)\}$ is a hyperedge of $\mathcal{H}'$.  Here $v_0,\ldots, v_k$ are distinct in $V$  while  $\varphi(v_0),\ldots,\varphi(v_k)$ may not be distinct in $V'$.  Such  a  morphism $\varphi:  \mathcal{H}\longrightarrow \mathcal{H}'$  induces  simplicial maps
$\delta\varphi:  \delta \mathcal{H}\longrightarrow \delta\mathcal{H}'$    and
$\Delta\varphi:  \Delta \mathcal{H}\longrightarrow \Delta\mathcal{H}'$
such that $\varphi=(\Delta\varphi)\mid_{\mathcal{H}}$ and $\delta\varphi=\varphi\mid_{\delta\mathcal{H}}$.
Suppose  in addition  that  $V\cap V'=\emptyset    $.   Let  $V\sqcup  V'$  be  the  disjoint  union.
  We define the {\it join}  $\mathcal{H}*\mathcal{H}'$    to  be  a   hypergraph  on    $V\sqcup  V'$    by
\begin{eqnarray*}
 \mathcal{H}*\mathcal{H}' =\{\sigma*\sigma'\mid \sigma\in\mathcal{H}{\rm ~and~}\sigma'\in\mathcal{H}'\}  \cup\mathcal{H}\cup\mathcal{H}'.
\end{eqnarray*}

\subsection{The  homology of  hypergraphs}\label{hypergraph-homology}

  For each $n\geq 0$,  let $(\Delta\mathcal{H})_n$  be the set of all the $n$-simplices in $\Delta\mathcal{H}$.  For  each $0\leq i\leq n$,
consider the face map
 $d_i:  (\Delta\mathcal{H})_n\longrightarrow (\Delta\mathcal{H})_{n-1}$
given by $d_i\{v_0,v_1\ldots,v_n\}= \{v_0,\ldots,    \widehat{v_i},\ldots, v_n\}$.
Let $R$ be a commutative ring with unit.  Let $C_n(\Delta\mathcal{H};R)$  be the  free $R$-module  generated by  all   the  elements  in  $(\Delta\mathcal{H})_n$.   Let  $\partial_n=\sum_{i=0}^n (-1)^i d_i$,  which  extends  linearly  over  $R$.    We  have a chain complex
\begin{eqnarray*}
\cdots   \overset{\partial_{n+1}}{ \longrightarrow}   C_{n} (\Delta\mathcal{H};R)\overset {\partial_{n }}{ \longrightarrow}  C_{n-1} (\Delta\mathcal{H};R)\overset{\partial_{n-1}}{ \longrightarrow} \cdots\overset {\partial_{1}}{ \longrightarrow}   C_{0} (\Delta\mathcal{H};R)\overset {\partial_{0}}{ \longrightarrow}  0,
\end{eqnarray*}
denoted  as $C_*(\Delta\mathcal{H};R)$.
 Let $R(\mathcal{H})_n$   be the collection of all the formal linear combinations of the $n$-hyperedges in $\mathcal{H}$ with coefficients in $R$.   Then $R(\mathcal{H})_*$ is a  graded sub-$R$-module  of $C_*(\Delta\mathcal{H}; R)$.  By \cite[Section~2]{h1},
the infimum chain complex
\begin{eqnarray*}
 \text{Inf}_n(R(\mathcal{H})_*)= R(\mathcal{H})_n\cap\partial_n^{-1}(R(\mathcal{H})_{n-1}),  \text{\ \ \ } n\geq 0
\end{eqnarray*}
 is    the largest sub-chain complex of $C_*(\Delta\mathcal{H}; R)$      contained in $R(\mathcal{H})_*$
  and
the supremum chain complex
 \begin{eqnarray*}
\text{Sup}_n(R(\mathcal{H})_*)= R(\mathcal{H})_n+\partial_{n+1}(R(\mathcal{H})_{n+1}), \text{\ \ \ }n\geq 0
\end{eqnarray*}
 is   the smallest sub-chain complex of $C_*(\Delta\mathcal{H}; R)$   containing   $R(\mathcal{H})_*$.
 The canonical inclusion
  $\iota: \text{Inf}_n(R(\mathcal{H})_*)\longrightarrow \text{Sup}_n(R(\mathcal{H})_*)  $
induces an isomorphism
\begin{eqnarray}\label{eq-0.1}
\iota_*:   H_n( \text{Inf}(R(\mathcal{H})_*))
 \overset{\cong}{\longrightarrow}H_n(\text{Sup}(R(\mathcal{H})_*)).
\end{eqnarray}
This  homology group  is  the  $n$-th  {\it embedded homology} of
  $\mathcal{H}$  and is  denoted  as $H_n(\mathcal{H};R)$ (cf. \cite[Subsection~3.2]{h1}).

Let $\varphi:\mathcal{H}\longrightarrow\mathcal{H}'$ be   a morphism  of hypergraphs.
We  have  a  commutative  diagram  of chain  complexes
 {\small\begin{eqnarray*}
 \xymatrix{
{\rm  Ker}((\Delta \varphi)_\#) \ar[r] & C_*(\Delta\mathcal{H}; R) \ar[r]^{(\Delta \varphi)_\#} &   C_*(\Delta\mathcal{H}'; R) \ar[r] &{\rm  Coker}((\Delta \varphi)_\#) \\
{\rm  Ker}(\text{Sup} (\varphi)) \ar[r]\ar[u]  &\text{Sup}_*(\mathcal{H}) \ar[u]\ar[r]^{\text{Sup} (\varphi)}& \text{Sup}_*(\mathcal{H}') \ar[u]   \ar[r] &{\rm  Coker}(\text{Sup} (\varphi))\ar[u]\\
{\rm  Ker}(\text{Inf} (\varphi)) \ar[r]\ar[u]& \text{Inf}_*(\mathcal{H}) \ar[u]\ar[r]^{\text{Inf} (\varphi)}& \text{Inf}_*(\mathcal{H}') \ar[u] \ar[r] &{\rm  Coker}(\text{Inf} (\varphi))\ar[u]\\
{\rm  Ker}((\delta \varphi)_\#) \ar[r]\ar[u]& C_*(\delta\mathcal{H}; R) \ar[u]\ar[r]^{(\delta \varphi)_\#} &   C_*(\delta\mathcal{H}'; R)
 \ar[u] \ar[r] &{\rm  Coker}((\delta \varphi)_\#)\ar[u],
 }
 \end{eqnarray*}}where  the  vertical arrows  in the first  three  columns  are  inclusions  and  the  vertical arrows  in the last column  are  homomorphisms of  chain complexes.
This   induces   a  commutative diagram of  homology groups
{\small\begin{eqnarray*}
 \xymatrix{
H_*(\text{Ker}((\Delta \varphi)_\#))\ar[r] & H_* (\Delta \mathcal{H}; R) \ar[r]^{(\Delta \varphi)_*} &   H_*(\Delta\mathcal{H}'; R) \ar[r] & H_*(\text{Coker}((\Delta  \varphi )_\#))\\
H_*(\text{Ker}(\text{Sup} (\varphi)))\ar[rd]\ar[u]& &
& H_*(\text{Coker}(\text{Sup} (\varphi)))\ar[u]\\
&H_*(\mathcal{H};R) \ar[uu]\ar[r]^{\varphi_*}& H_*(\mathcal{H}';R) \ar[uu]\ar[ru]\ar[rd]&\\
H_*(\text{Ker}(\text{Inf} (\varphi)))\ar[ru]\ar[uu]_{} &&
& H_*(\text{Coker}(\text{Inf} (\varphi)))\ar[uu]_{}\\
H_*(\text{Ker}((\delta \varphi)_\#))\ar[r]\ar[u]&  H_*(\delta \mathcal{H}; R) \ar[uu]\ar[r]^{(\delta  \varphi)_*} &   H_*(\delta\mathcal{H}'; R)\ar[uu] \ar[r]& H_*(\text{Coker}((\delta \varphi)_\#)).
 \ar[u]
 }
 \end{eqnarray*}}The  next  proposition follows from  \cite[Theorem~3.10]{h1}.

\begin{proposition}[Mayer-Vietoris sequences]
\label{pr-3.8897}
Let $\mathcal{H}$ and $\mathcal{H}'$ be two hypergraphs  satisfying   the  hypothesis
\begin{quote}
{\sc (P)}.  for any $\sigma\in\mathcal{H}$ and any  $\sigma'\in\mathcal{H}'$,  either  $\sigma\cap\sigma'$ is the  emptyset   or   $\sigma\cap\sigma'\in \mathcal{H}\cap\mathcal{H}'$.
\end{quote}
  Then  we  have a commutative diagram
{\small\begin{eqnarray*}
\xymatrix{
 \cdots\ar[r]
 & H_n(\delta\mathcal{H}\cap\delta\mathcal{H}')\ar[r]\ar[d]
 & H_n(\delta\mathcal{H})\oplus  H_n(\delta\mathcal{H}') \ar[r]\ar[d]
 &
&
& \\
 \cdots\ar[r]
 & H_n(\mathcal{H}\cap\mathcal{H}')\ar[r]\ar[d]
 & H_n(\mathcal{H})\oplus  H_n(\mathcal{H}') \ar[r]\ar[d]
 &
&
&\\
\cdots\ar[r]
 & H_n(\Delta\mathcal{H} \cap\Delta\mathcal{H}' )\ar[r]
 & H_n(\Delta\mathcal{H} )\oplus  H_n(\Delta\mathcal{H}' ) \ar[r]
 &
&
&
}\\
~~~\\
\xymatrix{
  &
 & \ar[r]
 &  H_n(\delta\mathcal{H}\cup\delta\mathcal{H}') \ar[r]\ar[d]
&   H_{n-1}(\delta\mathcal{H}\cap\delta\mathcal{H}')\ar[r]\ar[d]
& \cdots\\
 &
 & \ar[r]
 &  H_n(\mathcal{H}\cup\mathcal{H}') \ar[r]\ar[d]
&   H_{n-1}(\mathcal{H}\cap\mathcal{H}')\ar[r]\ar[d]
& \cdots\\
 &
 &  \ar[r]
 &  H_n(\Delta\mathcal{H} \cup\Delta\mathcal{H}' ) \ar[r]
&   H_{n-1}(\Delta\mathcal{H} \cap\Delta\mathcal{H}'  )\ar[r]
& \cdots
}
\end{eqnarray*}}such that each row is a long exact sequence    and each vertical map is a homomorphism    induced by the canonical inclusions  of  hypergraphs.
\end{proposition}

 The  next  proposition follows from  \cite[Theorem~1.1]{wrl}.

\begin{proposition}[K\"unneth-type formulae]
\label{pr-3.3.996}
Let $R$  be a principal ideal domain  with  unit $1$.
 We  have  a  commutative diagram
{\small\begin{eqnarray*}
\xymatrix{
  \bigoplus_{p+q+1=n} H_{p+1}(\delta\mathcal{H})\otimes H_{q+1}(\delta\mathcal{H}')\ar[r]\ar[d]
&H_{n+1}(\delta\mathcal{H} *\delta\mathcal{H}') \ar[r]\ar[d]
& &  \\
  \bigoplus_{p+q+1=n} H_{p+1}(\mathcal{H})\otimes H_{q+1}(\mathcal{H}')\ar[r]\ar[d]
&H_{n+1}(\mathcal{H} *\mathcal{H}')\ar[r]\ar[d]
 &  & \\
  \bigoplus_{p+q+1=n} H_{p+1}(\Delta\mathcal{H})\otimes H_{q+1}(\Delta\mathcal{H}')\ar[r]
&H_{n+1}(\Delta\mathcal{H} *\Delta\mathcal{H}')\ar[r]
 &  &
}\\
~~~\\
\xymatrix{
 &
&  \ar[r]
&\bigoplus_{p+q+1=n} {\rm  Tor}_R(H_{p+1}(\delta\mathcal{H}), H_{q}(\delta\mathcal{H}')) \ar[d]   \\
 &
& \ar[r]
 &\bigoplus_{p+q+1=n} {\rm  Tor}_R(H_{p+1}(\mathcal{H}), H_{q}(\mathcal{H}')) \ar[d]  \\
 &
 & \ar[r]
 &\bigoplus_{t+s+1=n} {\rm  Tor}_R(H_{p+1}(\Delta\mathcal{H}), H_{q}(\Delta\mathcal{H}'))
}
\end{eqnarray*}}such that each row  is  a  short  exact sequence  and  each  vertical  map  is  a  homomorphism        induced  by  the  canonical  inclusions  of  hypergraphs.
\end{proposition}

 \section{Simple-homotopy types  of hypergraphs  and  homology}\label{sect4}

\subsection{Collapses  of  hypergraphs}\label{ssssssss3.11111111}

Let $\mathcal{H}$  be a hypergraph.
 Let $\eta,\xi\in \mathcal{H}$.
  If $\eta \subsetneq  \xi$ and $\eta$ is not a proper face of any other hyperedges  in  $\mathcal{H}$,   then we say that $\eta$ is a {\it free face} of $\xi$  in $\mathcal{H}$.
Let  $\mathcal{A}\subseteq \mathcal{H}$  be  a  sub-hypergraph  of  $\mathcal{H}$.   We  call  $(\mathcal{H},\mathcal{A})$  a  {\it   pair}  of  hypergraphs.

\begin{definition}\label{def-88987}
\begin{enumerate}[(1).]
\item
 We say that
  there exists a   {\it weak elementary simplicial collapse} from $\mathcal{H}$ to $\mathcal{A}$  if
there exist  $\eta,\xi\in\mathcal{H}$  such that   $\eta$  is a free face of  $\xi$  in $\mathcal{H}$  and
$\mathcal{A}=\mathcal{H}\setminus\{\eta,\xi\}$  is  not   the  emptyset;

\item
  We say that
  there exists a   {\it strong elementary    simplicial collapse} from $\mathcal{H}$ to $\mathcal{A}$  if there exist  a
   weak
  elementary simplicial collapse  from $\mathcal{H}$ to $\mathcal{A}$
   such  that  (i).
$\mathcal{H}=\mathcal{A}\cup \Delta (\{v\}* \eta)$   and
(ii).
$\mathcal{A}\cap\Delta(\{v\}*\eta)=\{v\}*{\rm bd}(\eta)$    for  some $v\in V$ and  some $\eta\in\mathcal{H}$.
 Equivalently,   there exists a  strong elementary    simplicial collapse  from $\mathcal{H}$ to $\mathcal{A}$   if   both  (i)  and  (ii)     are  satisfied  such that
$\eta$  is  a free face  of  $\{v\} *\eta$   in   $\mathcal{H}$.
  \end{enumerate}
\end{definition}

\begin{definition}
We say that $\mathcal{H}$ {\it (weakly, strongly) collapses simplicially} to  $\mathcal{A}$   or $\mathcal{A}$  {\it (weakly, strongly) expands simplicially} to $\mathcal{H}$    if    there is a finite sequence of  (weak, strong)  elementary simplicial collapses $\mathcal{H}=\mathcal{H}_0\to \mathcal{H}_1\to \cdots\to \mathcal{H}_q = \mathcal{A}$.
\end{definition}

\begin{definition}
Let $\mathcal{H}$ and $\mathcal{H}'$  be two hypergraphs.
 We say that  $\mathcal{H}$  and $\mathcal{H}'$  have the same {\it  (weak, strong) simple-homotopy type}
  if  there is a finite sequence $\mathcal{H}=\mathcal{H}_0 \to \mathcal{H}_1\to\cdots\to \mathcal{H}_q=\mathcal{H}'$,   where each arrow represents a (weak, strong) simplicial expansion or a  (weak, strong)  simplicial collapse.
\end{definition}

\begin{example}
Let $\mathcal{H}=\{ \{v_0,v_1\}, \{v_0,v_1,v_2\},\{v_0,v_1,v_3\}, \{v_0,v_1,v_2,v_4,v_5\}\}$.
Then
   (i).  $\{v_0,v_1\}$  is not a free face of $\{v_0,v_1,v_2\}$   in $\mathcal{H}$,
     (ii).   $\{v_0,v_1,v_2\}$  is a   free face of  $\{v_0,v_1,v_2,v_4,v_5\}$  in $\mathcal{H}$,
    (iii).  $\{v_0,v_1\}$  is  not  a free face of  $\{v_0,v_1,v_2,v_4,v_5\}$  in $\mathcal{H}$.
\end{example}

\begin{example}
Let $(\mathcal{H},\mathcal{A})$ by
$
\mathcal{H}=\{\{v_0,v_1\},\{v_2,v_3\},\{v_0,v_1,v_2,v_3\}\}$   and   $\mathcal{A}=\{\{v_2,v_3\}\}$.
  Let $\eta=\{v_0,v_1\}$  and  $\xi=\{v_0,v_1,v_2,v_3\}$.   Then   $\eta$  is a free  face of   $\xi$  in $\mathcal{H}$  and   $\mathcal{A}=\mathcal{H}\setminus\{\eta,\xi\}$.  Thus  there is a weak elementary simplicial collapse from $\mathcal{H}$ to $\mathcal{A}$.  However,  there does not exist any  strong  elementary  simplicial collapse from $\mathcal{H}$ to $\mathcal{A}$.
\end{example}
\begin{example}
Let $(\mathcal{H},\mathcal{A})$  by
\begin{eqnarray*}
 \mathcal{H}&=&\{\{v_0\},\{v_1\},\{v_2\}, \{v_3\},\{v_0,v_1\},\{v_0,v_2\},\{v_1,v_2\},\\
&&\{v_0,v_1,v_2\},\{v_0,v_1,v_2,v_3\}\},\\
 \mathcal{A}&=&\{\{v_0\},\{v_1\},\{v_2\}, \{v_0,v_2\},\{v_1,v_2\},\{v_0,v_1,v_2,v_3\}\}.
\end{eqnarray*}
Let  $\eta=\{v_0,v_1\}$  and  $v=v_2$.   Then
\begin{eqnarray*}
&\Delta (\{v\}* \eta)= \{\{v_0\},\{v_1\},\{v_2\},  \{v_0,v_1\},\{v_0,v_2\},\{v_1,v_2\},\{v_0,v_1,v_2\}\},\\
&\{v\}*{\rm bd}(\eta)=\{\{v_0\},\{v_1\},\{v_2\}, \{v_0,v_2\},\{v_1,v_2\}\}.
\end{eqnarray*}
   Since both $\eta$  and $\{v\}*\eta$ are proper  faces of $\{v_0,v_1,v_2,v_3\}$,    $\eta$  is not a free face of $\{v\}*\eta$.  Thus   there does not exist any weak elementary simplicial collapse from $\mathcal{H}$ to $\mathcal{A}$.
\end{example}

\smallskip

   Let  $(\mathcal{K},\mathcal{A})$   be   a  pair  of   simplicial  complexes.  By    \cite[p.  3]{cohenm}   and  \cite[p.  106]{forman1},  the  weak  simple-homotopy  type  as  well as    the  strong    simple-homotopy  type
   is  the  same  as the  usual  simple-homotopy  type     (cf.  \cite[pp.  3-4]{cohenm}  and  \cite[p.  106]{forman1}).

\subsection{Collapses and the embedded homology}

 \begin{lemma}
\label{pr-2.3.main}
 Let $(\mathcal{H},\mathcal{A})$  be a  pair of hypergraphs  such that all the vertices are $0$-hyperedges.   If  there is  a  strong  elementary  simplicial collapse from $\mathcal{H}$ to $\mathcal{A}$,     then the
 canonical  inclusion  of   $\mathcal{A}$  into   $\mathcal{H}$   induces  an   isomorphism  from   $H_*(\mathcal{A})$  to   $H_*(\mathcal{H})$.
 \end{lemma}

 \begin{proof}
 Suppose  there is  a  strong   elementary simplicial collapse from $\mathcal{H}$ to $\mathcal{A}$.  Then   there exist a vertex $v$ of $\mathcal{H}$ and a hyperedge $\eta\in\mathcal{H}$  such that  (1).
$\mathcal{H}=\mathcal{A}\cup \Delta (\{v\}*  \eta)$,
 (2).
$\mathcal{A}\cap\Delta(\{v\}* \eta)=\{v\}* {\rm bd}(\eta)$,        and  (3). $\eta$  is a free face of $\{v\}* \eta$ in $\mathcal{H}$.   It follows from (1)  and (2)  that  $\mathcal{A}=\mathcal{H}\setminus  \{\eta, \{v\}* \eta\}$.       It follows  from (3)  that   $\{v\}* \eta$  is a maximal hyperedge in $\mathcal{H}$.
  In  (3),  we  substitute $\mathcal{H}$   with $\mathcal{A}$  and substitute  $\mathcal{H}'$  in   with $\Delta(\{v\}* \eta)$.   It   follows    that
 for any $\tau\in \mathcal{A} $  and any $\tau'\in \Delta(\{v\}* \eta)$,    either $\tau\cap\tau'$ is the  emptyset   or $\tau\cap\tau'$  is a proper subset of $\{v\}* \eta$ such  that  $\tau\cap\tau'$   is not equal to $\eta$.  Thus    either  $\tau\cap\tau'$ is the  emptyset   or $\tau\cap\tau'$ is  a simplex   of  $ \mathcal{A} \cap \Delta(\{v\}* \eta)$.
  The Mayer-Vietoris sequence   gives  a long exact sequence
\begin{eqnarray}\label{eq-2.3.b1}
&\cdots\longrightarrow H_n( \mathcal{A} \cap \Delta(\{v\}* \eta))\longrightarrow H_n(\mathcal{A})\oplus  H_n(\Delta(\{v\}* \eta)) \nonumber\\&
\longrightarrow  H_n(\mathcal{H}) \longrightarrow H_{n-1}( \mathcal{A} \cap \Delta(\{v\}* \eta))\longrightarrow \cdots
\end{eqnarray}
Since both  $\mathcal{A} \cap \Delta(\{v\}* \eta)$  and $    \Delta(\{v\}* \eta)$ are simplicial complexes  such  that
\begin{eqnarray*}
|\mathcal{A} \cap \Delta(\{v\}* \eta)|=|\{v\}* {\rm bd}(\eta)| \simeq  |\Delta(\{v\}*\eta)| \simeq  |v|,
\end{eqnarray*}
   $H_n( \mathcal{A} \cap \Delta(\{v\}* \eta))=H_n(\Delta(\{v\}* \eta))$  is zero for $n\geq 1$ and is  $R$  for $n=0$.  Therefore,   (\ref{eq-2.3.b1})  implies $H_n(\mathcal{H})\cong H_n(\mathcal{A})$  for any $n\geq 2$  as well as  an  exact sequence
\begin{eqnarray}\label{eq-2.3.012}
0\longrightarrow H_1(\mathcal{A})\longrightarrow H_1(\mathcal{H})\longrightarrow  R\longrightarrow H_0(\mathcal{A})\oplus  R\longrightarrow H_0(\mathcal{H})\longrightarrow 0.
\end{eqnarray}
Since all the vertices are $0$-hyperedges,  if we use $\mathcal{H}_i$    to denote the  sub-hypergraph of $\mathcal{H}$   consisting of all the $i$-hyperedges in $\mathcal{H}$,   then  $\mathcal{H}_0\cup \mathcal{H}_1$
  is a  $1$-dimensional  simplicial  complex.
  By \cite[Proposition~3.5]{h1},
\begin{eqnarray*}
H_0(\mathcal{H})= H_0(\mathcal{H}_0\cup\mathcal{H}_1)=R^{\oplus k},
\end{eqnarray*}
   where $k$  is the  number of the connected  components of $\mathcal{H}_0\cup \mathcal{H}_1$.
   Similarly, $\mathcal{A}_0\cup \mathcal{A}_1$  is a  $1$-dimensional  simplicial complex and
      \begin{eqnarray*}
H_0(\mathcal{A})= H_0(\mathcal{A}_0\cup\mathcal{A}_1)=R^{\oplus l},
\end{eqnarray*}
   where $l$  is the  number of the connected  components of $\mathcal{A}_0\cup \mathcal{A}_1$.
     The  canonical inclusion of $\mathcal{A}_0\cup\mathcal{A}_1$  into $\mathcal{H}_0\cup\mathcal{H}_1$   induces  a   homotopy equivalence   $ |\mathcal{A}_0\cup\mathcal{A}_1|  \simeq   |\mathcal{H}_0\cup\mathcal{H}_1|$.
Thus      $k=l$   and  the  canonical inclusion of $\mathcal{A}_0\cup\mathcal{A}_1$  into $\mathcal{H}_0\cup\mathcal{H}_1$   induces an  isomorphism
 $ H_0(\mathcal{A})=H_0(\mathcal{A}_0\cup\mathcal{A}_1)  \cong H_0(\mathcal{H}_0\cup\mathcal{H}_1)   =H_0(\mathcal{H})$.
     It follows that  the fifth arrow in (\ref{eq-2.3.012}) is surjective and has   its  kernel  $R$ ,  which implies that the forth arrow in (\ref{eq-2.3.012})  is injective and the third arrow in (\ref{eq-2.3.012})  is the zero-map.
 Therefore,  $H_1(\mathcal{A})\cong H_1(\mathcal{H})$.   Summarizing all the above,  we  have  $H_*(\mathcal{A})\cong H_*(\mathcal{H})$.
 \end{proof}

\begin{proposition}\label{pr-3.2.1a111}
Let   $\mathcal{H}$ and $\mathcal{H}'$   be  two  hypergraphs  with  all the vertices as $0$-hyperedges and  of  the same   strong  simple-homotopy type.   Then their embedded homology groups $H_*(\mathcal{H})$ and $H_*(\mathcal{H}')$  are  isomorphic.
\end{proposition}

\begin{proof}
  There is a finite sequence $\mathcal{H}=\mathcal{H}_0 \to \mathcal{H}_1\to\cdots\to \mathcal{H}_q=\mathcal{H}'$  of hypergraphs  where each arrow represents a strong simplicial expansion or a  strong simplicial collapse.    By  Proposition~\ref{pr-2.3.main},  for each $i=1,\ldots,q$,  the  embedded homology groups $H_*(\mathcal{H}_i)$ and $H_*(\mathcal{H}_{i-1})$  are  isomorphic.  Thus the  embedded homology groups $H_*(\mathcal{H})$ and $H_*(\mathcal{H}')$  are  isomorphic.
\end{proof}

\begin{example}
Consider the  hypergraphs
\begin{eqnarray*}
\mathcal{H}&=&\{\{v_0\}, \{v_1\}, \{v_2\}, \{v_3\}, \{v_0,v_1\}, \{ v_0,v_2\},  \{v_1,v_2\},\\
&& \{ v_0,v_1,v_2\},  \{v_1,v_2,v_3\}\},\\
\mathcal{H}'&=&\{\{v_1\}, \{v_2\}, \{v_3\},  \{v_1,v_2\},  \{v_1,v_2,v_3\}\}.
\end{eqnarray*}
Then  we  have    strong  elementary   simplicial  collapses  $\mathcal{H}\to  \mathcal{H}_1\to  \mathcal{H}'$,  where
\begin{eqnarray*}
\mathcal{H}_1 =  \{\{v_0\},  \{v_1\},  \{v_2\},  \{v_3\},    \{v_0,v_2\},   \{v_1,v_2\},    \{v_1,v_2,v_3\}\}.
\end{eqnarray*}
Thus  $\mathcal{H}$  and  $\mathcal{H}'$  have  the same    strong  simple-homotopy  type.   We  have
\begin{eqnarray*}
H_n(\mathcal{H})\cong  H_n(\mathcal{H}')=\begin{cases}
 R\oplus  R,  &n=0,\\
  0,  &n\geq   1.
\end{cases}
\end{eqnarray*}
\end{example}

\subsection{Collapses and   the homology of the associated simplicial complexes}\label{ss88.2}

 We  have  a  long exact sequence
\begin{eqnarray}
&
 \cdots\longrightarrow
 H_n(\Delta(\mathcal{H}\setminus \mathcal{A}))\longrightarrow
  H_n(\Delta \mathcal{H})\longrightarrow
 \tilde H_n(|\Delta\mathcal{H}|/|\Delta(\mathcal{H}\setminus \mathcal{A})|)\nonumber\\
 &\longrightarrow
 H_{n-1}(\Delta(\mathcal{H}\setminus \mathcal{A}))\longrightarrow
 \cdots,
 \label{eq-les2}
\end{eqnarray}
where
\begin{eqnarray*}
\tilde H_n(|\Delta\mathcal{H}|/|\Delta(\mathcal{H}\setminus \mathcal{A})|) = H_n (C_*(\Delta\mathcal{H})/ C_*(\Delta(\mathcal{H}\setminus \mathcal{A})))
\end{eqnarray*}
is  the  $n$-th  reduced homology group     of
 the quotient  topological  space.

\begin{lemma}\label{pr-x218}
Let $(\mathcal{H},\mathcal{A})$  be a  pair of hypergraphs.  If there is a weak  elementary simplicial collapse from $\mathcal{H}$ to $\mathcal{A}$,  then  the  canonical  inclusion  of   $\Delta\mathcal{A}$ into  $\Delta\mathcal{H}$   induces  an  isomorphism   from    $H_*(\Delta \mathcal{A})$  to    $H_*(\Delta\mathcal{H})$.
\end{lemma}
\begin{proof}
Suppose   $\mathcal{A}=\mathcal{H}\setminus\{\eta,\xi\}$,  where   $\eta$  is a free face of  $\xi$  in $\mathcal{H}$.
 Let  $\max(\mathcal{H})$  be the collection of all the maximal hyperedges
 in $\mathcal{H}$.  Then
\begin{eqnarray*}
\Delta(\mathcal{H}\setminus \{\eta, \xi\})=
\begin{cases}
\Delta\mathcal{H}\setminus \{\eta, \xi\},
 & {\rm~if~}  \xi\in\max(\mathcal{H}),\\
\Delta\mathcal{H},
& {\rm~if~}  \xi\notin\max(\mathcal{H}).
\end{cases}
\end{eqnarray*}
It follows that
$|\Delta\mathcal{H}/(\Delta\mathcal{H}\setminus \{\eta, \xi\})|\simeq  *$,
which implies
\begin{eqnarray}\label{eq-vns2}
\tilde H_n(\Delta\mathcal{H}/\Delta(\mathcal{H}\setminus \{\eta, \xi\}))=0
\end{eqnarray}
for each  $n\geq 0$.    In  (\ref{eq-les2}),  substitute   $\mathcal{A}$  with  $\{\eta,\xi\}$.  With the help of  (\ref{eq-vns2}),
\begin{eqnarray*}
H_n(\Delta\mathcal{A} )=H_n(\Delta(\mathcal{H}\setminus \{\eta, \xi\}))\cong  H_n(\Delta\mathcal{H}),
\end{eqnarray*}
where  the  isomorphism  is  induced  by  the canonical  inclusion  of  $\Delta\mathcal{A}$  into  $\Delta\mathcal{H}$.
\end{proof}

\begin{proposition}\label{pr-3.2.vf}
Let $\mathcal{H}$ and $\mathcal{H}'$  be  two  hypergraphs  of  the same  weak simple-homotopy type,  then
  the  homology groups  $H_*(\Delta\mathcal{H})$  and  $H_*(\Delta\mathcal{H}')$  are  isomorphic.
\end{proposition}

\begin{proof}
  There is a finite sequence $\mathcal{H}=\mathcal{H}_0 \to \mathcal{H}_1\to\cdots\to \mathcal{H}_q=\mathcal{H}'$  of  hypergraphs  where each arrow represents a weak  simplicial expansion or a  weak simplicial collapse.  By  Lemma~\ref{pr-x218},  for each $i=1,\ldots,q$,   $H_*(\Delta\mathcal{H}_i)$ and $H_*(\Delta\mathcal{H}_{i-1})$  are  isomorphic.  Thus     $H_*(\Delta\mathcal{H})$ and $H_*(\Delta\mathcal{H}')$  are  isomorphic.
\end{proof}

\begin{example}
Consider the  hypergraphs
\begin{eqnarray*}
\mathcal{H}& = &  \{ \{v_0\},  \{v_2\},  \{v_0,v_1,v_2\},  \{v_2,v_3,v_4\},  \{v_0,v_1,v_2,v_3,v_4\}\},\\
\mathcal{H}'&=&\{\{v_0\}\}.
\end{eqnarray*}
Then  we  have     weak  elementary  simplicial collapses  $\mathcal{H}   \to  \mathcal{H}_1\to  \mathcal{H}'$,   where
\begin{eqnarray*}
\mathcal{H}_1= \{\{v_0\},  \{v_2\},  \{v_0,v_1,v_2\}\}.
\end{eqnarray*}
Thus  $\mathcal{H}$  and  $\mathcal{H}'$  have  the same   weak simple-homotopy  type.
We  have   $\Delta\mathcal{H}=\Delta [v_0,v_1,v_2,v_3,v_4]$,   $\Delta\mathcal{H}'=\{\{v_0\}\}$  and
\begin{eqnarray*}
H_n(\Delta\mathcal{H})\cong  H_n(\Delta\mathcal{H}')=\begin{cases}
 G,  &n=0,\\
  0,  &n\geq   1.
\end{cases}
\end{eqnarray*}
\end{example}

\subsection{Collapses and   the   homology of  The lower-associated simplicial complexes}

  We  have  a  long exact sequence
\begin{eqnarray}
 &\cdots\longrightarrow
  H_n(\delta(\mathcal{H}\setminus \mathcal{A}))\longrightarrow
   H_n(\delta \mathcal{H})\longrightarrow
  \tilde H_n(|\delta\mathcal{H}|/|\delta(\mathcal{H}\setminus \mathcal{A})|)\nonumber\\
  &\longrightarrow
  H_{n-1}(\delta(\mathcal{H}\setminus \mathcal{A}))\longrightarrow
  \cdots.
  \label{eq-les1}
\end{eqnarray}

\begin{lemma}\label{pr-2.8a182}
Let $(\mathcal{H},\mathcal{A})$  be a  pair of hypergraphs.  Suppose there is a  weak   elementary simplicial collapse from $\mathcal{H}$ to $\mathcal{A}$,  say  $\mathcal{A}=\mathcal{H}\setminus\{\eta,\xi\}$  where $\eta,  \xi \in \mathcal{H}$  such that $\eta$  is a free face  of $\xi$.
\begin{enumerate}[(i).]
\item
If one of the followings  is satisfied:
\begin{enumerate}[(a).]
\item
$\dim \eta=0$,
\item
$\sigma\in\mathcal{H}$ for any  nonempty  subset $\sigma\subseteq \xi$,
\item
there exists a  nonempty  subset $\tau\subseteq \eta$ such that $\tau\notin \mathcal{H}$,
\end{enumerate}
then  the  canonical  inclusion  of  $\delta\mathcal{A}$  into  $\delta\mathcal{H}$ induces an isomorphism   $H_n(\delta\mathcal{A})\cong H_n(\delta\mathcal{H})$ for any $n\geq 0$;

\item
If all the followings are satisfied:
\begin{enumerate}[(a)'.]
\item
$\dim\eta\geq  1$,
\item
there exists a  nonempty  subset $\sigma\subseteq \xi$  such that  $\sigma\notin \mathcal{H}$,
\item
$\tau\in\mathcal{H}$ for any  nonempty   subset   $\tau\subseteq\eta$,
\end{enumerate}
then
\begin{enumerate}[(1).]
\item
 the  canonical  inclusion  of  $\delta\mathcal{A}$  into  $\delta\mathcal{H}$ induces an isomorphism $H_n(\delta\mathcal{A})\cong H_n(\delta\mathcal{H})$ for any  $n\neq  \dim\eta, \dim\eta-1$,
\item
there  is  an  injection from $H_{\dim \eta}(\delta\mathcal{A})$  into  $H_{\dim \eta}(\delta\mathcal{H})$  whose cokernel  is a subgroup of  $R$ ,
\item
there  is  an  surjection from $H_{\dim \eta-1}(\delta\mathcal{A})$  into  $H_{\dim \eta-1}(\delta\mathcal{H})$  whose kernel  is a quotient group of  $R$ .
\end{enumerate}
\end{enumerate}
\end{lemma}

\begin{proof}
Suppose   $\mathcal{A}=\mathcal{H}\setminus\{\eta,\xi\}$,  where   $\eta$  is a free face of  $\xi$  in $\mathcal{H}$.   In  (\ref{eq-les1}),  substitute   $\mathcal{A}$  with  $\{\eta,\xi\}$.  We  obtain a long exact sequence
\begin{eqnarray}\label{eq-0688}
&\cdots\longrightarrow
 H_n(\delta(\mathcal{H}\setminus \{\eta,\xi\}))\longrightarrow
  H_n(\delta \mathcal{H})\longrightarrow
 \tilde H_n(\delta\mathcal{H}/\delta(\mathcal{H}\setminus \{\eta,\xi\}))\nonumber\\
 &\longrightarrow
 H_{n-1}(\delta(\mathcal{H}\setminus \{\eta,\xi\}))\longrightarrow\cdots.
\end{eqnarray}

{\sc Case~1}.  $\sigma\in\mathcal{H}$  for  any   nonempty   subset  $\sigma\subseteq\xi$.

Then  $\xi\in\delta\mathcal{H}$.  Since $\eta$  is not a face of any hyperedge $\kappa\neq \xi$  in $\mathcal{H}$,  it follows that  $\eta$  is not a face of any simplex $\kappa\neq \xi$  in $\delta\mathcal{H}$.   Thus  for any $\kappa\in\delta\mathcal{H}$  such that $\kappa \neq \eta,\xi$  we  have $\kappa\in \delta(\mathcal{H}\setminus \{\eta, \xi\})$.
Hence   $\delta(\mathcal{H}\setminus \{\eta, \xi\})=\delta\mathcal{H}\setminus \{\eta, \xi\}$.   Therefore,
\begin{eqnarray*}
|\delta\mathcal{H}/(\delta\mathcal{H}\setminus \{\eta, \xi\})|\simeq  *.
\end{eqnarray*}

{\sc Case~2}.    there  exists a   nonempty   subset  $\sigma\subseteq\xi$
  such  that  $\sigma\notin\mathcal{H}$.

Then  $\xi\notin  \delta\mathcal{H}$.

{\sc  Subcase~2.1}.  $\tau\in\mathcal{H}$  for   any    nonempty   subset    $\tau\subseteq\eta$.

Then  $\eta\in  \delta\mathcal{H}$.  Since $\xi\notin  \delta\mathcal{H}$  and $\eta$  is not a proper face of any  hyperedge  $\tau\neq\xi$,  it follows that $\eta$    is  a maximal simplex   of $\delta\mathcal{H}$.    Similar with the argument of {\sc Case~1} we  have    $\delta(\mathcal{H}\setminus \{\eta, \xi\})=\delta\mathcal{H}\setminus \{\eta\}$.
Therefore,
\begin{eqnarray}\label{eq-1868ab}
|\delta\mathcal{H}/(\delta\mathcal{H}\setminus \{\eta, \xi\})|\simeq
\begin{cases}
   S^{\dim \eta},  &{\rm~if~}\dim \eta\geq 1,\\
   *,  &{\rm~if~}\dim\eta=0.
   \end{cases}
\end{eqnarray}

{\sc  Subcase~2.2}.  there exists  a   nonempty   subset  $\tau\subseteq\eta$
  such that  $\tau\notin\mathcal{H}$.

Then  $\eta\notin  \delta\mathcal{H}$.   Similar with the argument of {\sc Case~1} we  have    $\delta(\mathcal{H}\setminus \{\eta, \xi\})=\delta\mathcal{H} $.  Therefore,
\begin{eqnarray*}
|\delta\mathcal{H}/(\delta\mathcal{H}\setminus \{\eta, \xi\})|= *.
\end{eqnarray*}

(i).  Suppose one of (a),  (b)   and  (c)  is  satisfied.   By  summarizing the above  cases,
\begin{eqnarray}\label{eq-vns1}
\tilde H_n(\delta\mathcal{H}/\delta(\mathcal{H}\setminus \{\eta, \xi\}))=0
\end{eqnarray}
for each  $n\geq  0$.    By  (\ref{eq-0688})  and   (\ref{eq-vns1})  we obtain
\begin{eqnarray*}
H_n(\delta(\mathcal{H}\setminus \{\eta, \xi\}))\cong  H_n(\delta\mathcal{H}).
\end{eqnarray*}

(ii).    Suppose all  of (a)',  (b)'   and  (c)'  are  satisfied.   By (\ref{eq-1868ab}),
\begin{eqnarray*}
|\delta\mathcal{H}/(\delta\mathcal{H}\setminus \{\eta, \xi\})|\simeq
   S^{\dim \eta}.
\end{eqnarray*}
With the help of  (\ref{eq-0688}),
\begin{eqnarray}\label{eq-fafa.1}
H_n(\delta\mathcal{H})\cong H_n(\delta\mathcal{A})
\end{eqnarray}
for any  $n\neq  \dim\eta, \dim\eta-1$  and an  exact sequence
\begin{eqnarray}
&0\longrightarrow
 H_ {\dim\eta}(\delta(\mathcal{H}\setminus \{\eta,\xi\}))\longrightarrow
  H_{ \dim\eta}(\delta \mathcal{H})\longrightarrow
 R \nonumber\\
&\longrightarrow  H_ {\dim\eta-1}(\delta(\mathcal{H}\setminus \{\eta,\xi\}))\longrightarrow
  H_{ \dim\eta-1}(\delta \mathcal{H})\longrightarrow  0.
  \label{eq-fafa.2}
\end{eqnarray}
Therefore,  (1) follows from (\ref{eq-fafa.1}),  (2) follows from the second arrow of (\ref{eq-fafa.2}),  and (3) follows from the  fifth  arrow of (\ref{eq-fafa.2}).
\end{proof}

\begin{proposition}\label{pr-3.2.1a}
 Let    $\mathcal{H}$ and $\mathcal{H}'$  be    two hypergraphs    of   the same  weak simple-homotopy type via a sequence of weak simplicial collapses and  weak simplicial expansions  such that   in each  weak elementary  simplicial collapse or weak elementary  simplicial expansion,  $\eta$  is a free face of $\xi$ and one of (a), (b)  and (c) in Lemma~\ref{pr-2.8a182}  is satisfied.  Then  the  homology  groups    $H_*(\delta\mathcal{H} )$   and   $H_n(\delta \mathcal{H}')$  are  isomorphic.
\end{proposition}

\begin{proof}

   There is a finite sequence $\mathcal{H}=\mathcal{H}_0 \to \mathcal{H}_1\to\cdots\to \mathcal{H}_q=\mathcal{H}'$  of  hypergraphs  where each arrow represents a weak  simplicial expansion or a  weak simplicial collapse.
In  each  weak elementary  simplicial collapse or weak elementary  simplicial expansion,  $\eta$  is a free face of $\xi$ and one of (a), (b)  and (c) in Lemma~\ref{pr-2.8a182}  is satisfied.
Thus by Lemma~\ref{pr-2.8a182}~(i),
  for each $i=1,\ldots,q$,   $H_n(\delta\mathcal{H}_i)$ and $H_n(\delta\mathcal{H}_{i-1})$  are  isomorphic.
   Thus   $H_n(\delta\mathcal{H})$ and $H_n(\delta\mathcal{H}')$  are  isomorphic.
\end{proof}

 \begin{example}
Consider  the  hypergraphs
\begin{eqnarray*}
\mathcal{H}&=&\{ \{v_0\},  \{v_1\}, \{v_3\}, \{v_4\}, \{v_0,v_3\},  \{v_0,v_4\},  \{v_3,v_4\},\\
&&\{v_0,v_1,v_2\}, \{v_0,v_3,v_4\}, \{v_0,v_2,v_5\},  \{v_0,v_2,v_5,v_6,v_7\} \},\\
\mathcal{H}'&=& \{ \{v_0\}\}.
\end{eqnarray*}
Then  we  have     weak  elementary  simplicial collapses  $\mathcal{H}  \to  \mathcal{H}_1\to  \mathcal{H}_2\to  \mathcal{H}_3 \to  \mathcal{H}_4\to  \mathcal{H}'$,    where
\begin{eqnarray*}
\mathcal{H}_1=  \mathcal{H} \setminus  \{\{v_0,v_2,v_5\},  \{v_0,v_2,v_5,v_6,v_7\}\}
\end{eqnarray*}
with   $\{v_0,v_2,v_5\}$   a  free  face  of   $ \{v_0,v_2,v_5,v_6,v_7\} $  satisfying   Lemma~\ref{pr-2.8a182}~(c),
\begin{eqnarray*}
\mathcal{H}_2=  \mathcal{H}_1 \setminus  \{\{v_3,v_4\},  \{v_0,v_3,v_4 \}\}
\end{eqnarray*}
 with   $\{v_3,v_4\}$     a  free  face  of   $ \{v_0,v_3,v_4\} $  satisfying   Lemma~\ref{pr-2.8a182}~(b),
\begin{eqnarray*}
\mathcal{H}_3=  \mathcal{H}_2 \setminus  \{\{v_1\}, \{v_0,v_1,v_2\} \}
\end{eqnarray*}
with   $\{v_1\}$      a  free  face  of   $ \{v_0,v_1,v_2\} $  satisfying   Lemma~\ref{pr-2.8a182}~(a),  and
\begin{eqnarray*}
\mathcal{H}_4 =\mathcal{H}_3\setminus \{\{v_3\}, \{v_0,v_3\}\}, ~~~~~~
\mathcal{H}' =  \mathcal{H}_4\setminus \{\{v_4\}, \{v_0,v_4\}\}
\end{eqnarray*}
are  elementary simplicial collapses of  simplicial complexes.   We  have   $\delta\mathcal{H}=\Delta[v_0,v_3,v_4]$,    $\delta\mathcal{H}'=\{\{v_0\}\}$   and
\begin{eqnarray*}
H_n(\delta\mathcal{H})\cong  H_n(\delta\mathcal{H}')=\begin{cases}
 G,  &n=0,\\
  0,  &n\geq   1.
\end{cases}
\end{eqnarray*}
\end{example}

 \section{The   stability  of  persistent homology  of  hypergraphs}

\subsection{Filtrations of  hypergraphs and interleavings of  the  persistent   homology}

Let $\mathcal{H}$ be a hypergraph.
 A  {\it filtration} $\{\mathcal{H}_t\}_{t\in\mathbb{R}}$  of $\mathcal{H}$ is a family $\mathcal{H}_t$, $t\in\mathbb{R}$, of  sub-hypergraphs of $\mathcal{H}$  such that for any $t\leq s$,  there is an inclusion $i_{t,s}:  \mathcal{H}_t\longrightarrow \mathcal{H}_s$ such that $i_{s,r}\circ i_{t,s}=i_{t,r}$.  
 Let $f$  be a  function on $\mathcal{H}$ assigning a real number $f(\sigma)$  to each $\sigma\in\mathcal{H}$.
  For each $t\in \mathbb{R}$,  let  the level hypergraph be
  \begin{eqnarray}\label{eq-level-1}
   \mathcal{H}_{t}^f:=f^{-1}((-\infty, t])=\{\sigma\in\mathcal{H}\mid  f(\sigma)\leq t\}.
  \end{eqnarray}
Let $t\leq s$.    Let $i_t^s(f): \mathcal{H}_t^f \longrightarrow \mathcal{H}_s^f $   be the canonical inclusion  of hypergraphs,  which  gives a  filtration   $\mathcal{H}^f:=\{\mathcal{H}^f_t\}_{t\in\mathbb{R}}$  of $\mathcal{H}$.
 We have the induced inclusions
$\delta(i_t^s(f)):  \delta( \mathcal{H}_t^f)\longrightarrow  \delta( \mathcal{H}_s^f)$   and
 $\Delta(i_t^s(f)):  \Delta( \mathcal{H}_t^f)\longrightarrow  \Delta( \mathcal{H}_s^f)$
  of simplicial complexes,  which   give
   the  filtration   $\{\Delta( \mathcal{H}_t^f)\}_{t\in\mathbb{R}}$  of $\Delta\mathcal{H}$   and   the  filtration
   $\{\delta( \mathcal{H}_t^f)\}_{t\in\mathbb{R}}$  of $\delta\mathcal{H}$  respectively.
Let  $\mathbb{F}$  be   a  field.      We  have three persistent  vector spaces
  \begin{enumerate}[(1).]
  \item
 $ \{H_n(\mathcal{H}_r^f)\}_{r\in\mathbb{R}}:=\{H_*(\mathcal{H}_r^f;\mathbb{F}), \big (i_t^s(f)\big)_*\mid  r,t,s\in\mathbb{R},  t\leq s\}$;
  \item
  $\{H_n(\delta(\mathcal{H}_r^f))\}_{r\in\mathbb{R}}:=\{H_*(\delta(\mathcal{H}_r^f);\mathbb{F}), \big (\delta(i_t^s(f))\big)_*  \mid  r,t,s\in\mathbb{R},  t\leq s\}$;
 \item
  $\{H_n(\Delta(\mathcal{H}_r^f))\}_{r\in\mathbb{R}}:=\{H_*(\Delta(\mathcal{H}_r^f);\mathbb{F}), \big (\Delta(i_t^s(f))\big)_*  \mid  r,t,s\in\mathbb{R},  t\leq s\}$
  \end{enumerate}
whose   persistent diagrams  are  respectively
  \begin{enumerate}[(1)'.]
 \item
 $D(\mathcal{H}, f) :=  D(\{H_n(\mathcal{H}_t^f)\}_{t\in\mathbb{R}})$;
 \item
 $D(\delta\mathcal{H}, f) :=  D(\{H_n(\delta(\mathcal{H}_t^f))\}_{t\in\mathbb{R}})$;
 \item
$ D(\Delta\mathcal{H}, f) :=  D(\{H_n(\Delta(\mathcal{H}_t^f))\}_{t\in\mathbb{R}})$.
  \end{enumerate}

 Let $f,g: \mathcal{H}\longrightarrow \mathbb{R}$ be  two real valued functions on $\mathcal{H}$.     The $L^\infty$-distance between $f$ and $g$  is  defined  by
\begin{eqnarray*}
||f-g||_\infty= \sup_{\sigma}  |f(\sigma)-g(\sigma)|
\end{eqnarray*}
where $\sigma$ ranges over all hyperedges of $\mathcal{H}$.
Let  $d_B^\infty$  be  the bottleneck distance  (cf.  \cite{pd1,pd2}).

  \begin{theorem}
 \label{th-main111-x}
 Let $\mathcal{H}$ be a hypergraph. Let $f$ and $g$  be real-valued functions on $\mathcal{H}$. Then
  $d_B^\infty(D\mathcal{H}^f, D\mathcal{H}^g)\leq  ||f-g||_{\infty}$.   Here
  \begin{eqnarray*}
   d_B^\infty(D\mathcal{H}^f, D\mathcal{H}^g)&=& \max
\{ d_B^\infty (D(\mathcal{H}, f) , D(\mathcal{H}, g)), \\
&&  d_B^\infty (D(\delta\mathcal{H}, f), D(\delta\mathcal{H}, g)), \\
  &&  d_B^\infty (D(\Delta\mathcal{H}, f), D(\Delta\mathcal{H}, g) )\}.
  \end{eqnarray*}
 \end{theorem}

 \subsubsection{Proof of  Theorem~\ref{th-main111-x}}

   \begin{lemma}\label{le-888.1}
    Let $n\geq 0$.  Suppose
\begin{eqnarray}\label{eq-pd-1}
|| f-g||_{\infty}\leq\epsilon
\end{eqnarray}
 for some $\epsilon >0$.     Then the following   persistent vector spaces are strongly $\epsilon$-interleaved:
 \begin{enumerate}[(i).]
 \item
 the persistent embedded homology $\{H_n(\mathcal{H}_t^f)\}_{t\in\mathbb{R}}$ and  $\{H_n(\mathcal{H}_t^g)\}_{t\in\mathbb{R}}$;
 \item
 the persistent homology  $\{H_n(\delta(\mathcal{H}_t^f))\}_{t\in\mathbb{R}}$ and  $\{H_n(\delta(\mathcal{H}_t^g))\}_{t\in\mathbb{R}}$  of the associated simplicial complexes;
 \item
 the persistent homology   $\{H_n(\Delta(\mathcal{H}_t^f))\}_{t\in\mathbb{R}}$ and  $\{H_n(\Delta(\mathcal{H}_t^g))\}_{t\in\mathbb{R}}$ of  the lower-associated simplicial complexes.
 \end{enumerate}
  \end{lemma}

  \begin{proof}

    For any   $t\in \mathbb{R}$  there are  inclusions
  \begin{eqnarray*}
\varphi_t:  \mathcal{H}_{t}^f\longrightarrow  \mathcal{H}_{t+\epsilon}^g, ~~~~~~\psi_t:     \mathcal{H}_{t}^g\longrightarrow  \mathcal{H}_{t+\epsilon}^f
  \end{eqnarray*}
  of  hypergraphs.  By a diagram chasing,
  \begin{eqnarray*}
   \psi_s \circ  i_t^s(g) \circ  \varphi_{t-\epsilon} = i_{t-\epsilon}^{s+\epsilon}(f),  ~~~~~~i_{t+\epsilon}^{s+\epsilon}(f)\circ \psi_t= \psi_s\circ   i_t^s (g),\\
   \varphi_s\circ  i_t^s(f)\circ \psi_{t-\epsilon}={i}_{t-\epsilon}^{s+\epsilon}(g),~~~~~~i_{t+\epsilon}^{s+\epsilon}(f)\circ \varphi_t=\varphi_s\circ i_t^s(f).
  \end{eqnarray*}
   Taking the embedded homology,   we   have  that   for any   $t\in \mathbb{R}$  there are  homomorphisms
    \begin{eqnarray*}
(\varphi_t)_*:  H_*(\mathcal{H}_{t}^f;\mathbb{F})\longrightarrow  H_*(\mathcal{H}_{t+\epsilon}^g;\mathbb{F}), ~~~~~~\psi_t:     H_*(\mathcal{H}_{t}^g;\mathbb{F})\longrightarrow  H_*(\mathcal{H}_{t+\epsilon}^f;\mathbb{F}).
  \end{eqnarray*}
     Moreover,
  \begin{eqnarray*}
   (\psi_s)_* \circ  (i_t^s(g))_* \circ  (\varphi_{t-\epsilon})_* = (i_{t-\epsilon}^{s+\epsilon}(f))_*,  ~~~~~~(i_{t+\epsilon}^{s+\epsilon}(f))_*\circ (\psi_t)_*= (\psi_s)_*\circ   (i_t^s (g))_*,\\
   (\varphi_s)_*\circ  (i_t^s(f))_*\circ (\psi_{t-\epsilon})_*=({i}_{t-\epsilon}^{s+\epsilon}(g))_*,~~~~~~(i_{t+\epsilon}^{s+\epsilon}(f))_*\circ (\varphi_t)_*=(\varphi_s)_*\circ (i_t^s(f))_*.
  \end{eqnarray*}
By   \cite[Definition~4.2]{pd2},  we obtain (i).
 The proofs for  (ii)  and (iii)   are  similar.
  \end{proof}

 \begin{proof}[Proof of  Theorem~\ref{th-main111-x}]
 By    \cite[Theorem~4.4]{pd2}   and  Lemma~\ref{le-888.1}~(a),      $d_B^\infty  (D(\mathcal{H}, f) , D(\mathcal{H}, g) )\leq  ||f-g||_{\infty}$.    By    \cite[Theorem~4.4]{pd2}  and  Lemma~\ref{le-888.1}~(b),    $d_B^\infty  (D(\delta\mathcal{H}, f) , D(\delta\mathcal{H}, g) )\leq  ||f-g||_{\infty}$.   By    \cite[Theorem~4.4]{pd2}  and  Lemma~\ref{le-888.1}~(c),     $d_B^\infty  (D(\Delta\mathcal{H}, f) , D(\Delta\mathcal{H}, g) )\leq  ||f-g||_{\infty}$.
 \end{proof}

 \subsection{Persistent  homology for  morphisms between  hypergraphs  and  the stability}\label{s666}

 Let $\mathcal{H}$ and $\mathcal{H}'$ be two hypergraphs on  $V$ and $V'$ respectively.  Let $\varphi: \mathcal{H}\longrightarrow \mathcal{H}'$ be a morphism of hypergraphs.

(A).
Suppose $\{\mathcal{H}'_t\}_{t\in \mathbb{R}}$ is a filtration of $\mathcal{H}'$. We have a {\it pull-back filtration} $\{\varphi^\star\mathcal{H}'_t\}_{t\in \mathbb{R}}$ of $\mathcal{H}$ induced from $\varphi$, where for each $t\in \mathbb{R}$,
\begin{eqnarray*}
\varphi^\star \mathcal{H}'_t=\{\sigma\in\mathcal{H}\mid \varphi(\sigma)\in \mathcal{H}'_t\}.
\end{eqnarray*}
we have    morphisms  of hypergraphs
\begin{eqnarray*}
\varphi^\star_t:   \varphi^\star \mathcal{H}'_t \longrightarrow \mathcal{H}'_t, ~~~t\in\mathbb{R},
\end{eqnarray*}
where for each $t\in \mathbb{R}$, $\varphi^\star _t$ is the restriction of $\varphi$ to $\varphi^\star \mathcal{H}'_t$.
Consequently, we have a commutative diagram of  persistent  homology
{\small\begin{eqnarray*}
 \xymatrix{
H_*(\text{Ker}(\Delta (\varphi^\star _t)_\#))\ar[r] & H_* (\Delta(\varphi^\star \mathcal{H}'_t); R) \ar[r]^{\Delta (\varphi^\star _t)_*} &   H_*(\Delta(\mathcal{H}'_t); R) \ar[r] & H_*(\text{Coker}(\Delta (\varphi^\star _t)_\#))\\
H_*(\text{Ker}(\text{Sup} (\varphi^\star _t)))\ar[rd]\ar[u]&
&
& H_*(\text{Coker}(\text{Sup} (\varphi^\star _t)))\ar[u]\\
&H_*(\varphi^\star \mathcal{H}'_t) \ar[uu]\ar[r]^{(\varphi^\star _t)_*}& H_*(\mathcal{H}'_t) \ar[uu]\ar[ru]\ar[rd]&\\
H_*(\text{Ker}(\text{Inf} (\varphi^\star _t)))\ar[ru]\ar[uu]_{(\iota_t\mid_{\text{Ker}})_*} & & & H_*(\text{Coker}(\text{Inf} (\varphi^\star _t)))\ar[uu]_{((\iota'_t)^{\text{Coker}})_*}\\
H_*(\text{Ker}(\delta (\varphi^\star _t)_\#))\ar[r]\ar[u]&  H_*(\delta(\varphi^\star \mathcal{H}'_t); R) \ar[uu]\ar[r]^{\delta (\varphi^\star _t)_*} &   H_*(\delta(\mathcal{H}'_t); R)\ar[uu] \ar[r]& H_*(\text{Coker}(\delta (\varphi^\star _t)_\#)).
 \ar[u]
 }
 \end{eqnarray*}}

 (B).
Suppose $\{\mathcal{H}_t\}_{t\in \mathbb{R}}$ is a filtration of $\mathcal{H}$. We have a {\it push-forward filtration} $\{\varphi_\star\mathcal{H}_t\}_{t\in \mathbb{R}}$ of $\mathcal{H}'$ induced from $\varphi$, where for each $t\in \mathbb{R}$,
\begin{eqnarray*}
\varphi_\star \mathcal{H}_t=\{\varphi(\sigma)\mid \sigma\in \mathcal{H}_t\}.
  \end{eqnarray*}
 we have    morphisms  of   hypergraphs
\begin{eqnarray*}
\varphi_{\star  t}:   \mathcal{H}_t \longrightarrow \varphi_{\star}\mathcal{H}_t , ~~~t\in\mathbb{R},
\end{eqnarray*}
  where for each $t\in \mathbb{R}$, $\varphi_{\star  t}$ is the restriction of $\varphi$ to $\mathcal{H}_t$.
Consequently, we have a commutative diagram of  persistent  homology
{\small\begin{eqnarray*}
 \xymatrix{
H_*(\text{Ker}(\Delta (\varphi_{\star  t})_\#))\ar[r] & H_*(\Delta(\mathcal{H}_t); R) \ar[r]^{\Delta (\varphi_{\star  t})_*} &   H_*(\Delta(\varphi_{*}\mathcal{H}_t); R) \ar[r] & H_*(\text{Coker}(\Delta (\varphi_{\star  t})_\#))\\
H_*(\text{Ker}(\text{Sup} (\varphi_{\star  t})))\ar[rd]\ar[u]&  &
& H_*(\text{Coker}(\text{Sup} (\varphi_{\star  t})))\ar[u]\\
&H_*(\mathcal{H}_t) \ar[uu]\ar[r]^{(\varphi_{\star  t})_*}& H_*(\varphi_{\star}\mathcal{H}_t )\ar[uu]\ar[ru]\ar[rd]&\\
H_*(\text{Ker}(\text{Inf} (\varphi_{\star  t})))\ar[ru]\ar[uu]_{(\iota_t\mid_{\text{Ker}})_*} & &
& H_*(\text{Coker}(\text{Inf} (\varphi_{\star  t})))\ar[uu]_{((\iota'_t)^{\text{Coker}})_*}\\
H_*(\text{Ker}(\delta (\varphi_{\star  t})_\#))\ar[r]\ar[u]&  H_*(\delta(\mathcal{H}_t); R) \ar[uu]\ar[r]^{\delta (\varphi_{\star  t})_*} &   H_*(\delta(\varphi_{\star}\mathcal{H}_t ); R)\ar[uu] \ar[r]& H_*(\text{Coker}(\delta (\varphi_{\star  t})_\#)).
 \ar[u]
 }
 \end{eqnarray*}}

 Let
 $f,g: \mathcal{H}\longrightarrow \mathbb{R}$
  and  let
  $f',g': \mathcal{H}'\longrightarrow \mathbb{R}$.   We have    filtrations
 $\mathcal{H}_{t}^f=f^{-1}((-\infty, t])$  and   $\mathcal{H}^g_t=g^{-1}((-\infty, t])$
of $\mathcal{H}$  as  well as     filtrations
 ${\mathcal{H}'}_{t}^{f'}={f'}^{-1}((-\infty, t])$  and    ${\mathcal{H}'}^{g'}_t={g'}^{-1}((-\infty, t])$,  $t\in \mathbb{R}$,
 of $\mathcal{H}'$.
Let  $\varphi: \mathcal{H}\longrightarrow \mathcal{H}'$  be a morphism of hypergraphs.
We have the pull-back filtrations
$\varphi^\star (  {\mathcal{H}'}_{t}^{f'})$   and   $\varphi^\star (  {\mathcal{H}'}_{t}^{g'})$
 of  $\mathcal{H}$  as  well as  the push-forward filtrations
$\varphi_\star(  {\mathcal{H}}_{t}^{f})$   and   $ \varphi_\star(  {\mathcal{H}}_{t}^{g})$
 of   $\mathcal{H}'$.

 \begin{theorem}\label{main2-x}
Let $\epsilon>0$.   Let $f$ and $g$  be real-valued functions on $\mathcal{H}$.    Let  $f'$ and $g'$ be real-valued functions on $\mathcal{H}'$.
    \begin{enumerate}[(i).]
\item
If $||f-g||_\infty\leq \epsilon$,  then each of the  persistent  linear map among the twenty-one   persistent  linear maps in the commutative diagram in (B), denoted as $\Phi_f$ and $\Phi_g$, satisfies
 \begin{eqnarray*}
  d_B^\infty(D(\Phi_f), D(\Phi_g))\leq \epsilon.
  \end{eqnarray*}
Here
\begin{eqnarray*}
 d_B^\infty(D(\Phi_f), D(\Phi_g))&=& \max\{d_B^\infty(D(\text{Ker}(\Phi_f)), D(\text{Ker}(\Phi_g))),  \\ &&
 d_B^\infty(D(\text{Im}(\Phi_f)),  D(\text{Im}(\Phi_g))),\\ &&
  d_B^\infty(D(\text{Coker}(\Phi_f)), D(\text{Coker}(\Phi_g))) \};
 \end{eqnarray*}

\item
If $||f'-g'||_\infty\leq \epsilon$,   then each of the  persistent  linear map among the twenty-one    persistent  linear maps in the commutative diagram in (A), denoted as $\Phi'_{f'}$ and $\Phi'_{g'}$, satisfies
 \begin{eqnarray*}
  d_B^\infty(D(\Phi'_{f'}), D(\Phi'_{g'}))\leq \epsilon.
  \end{eqnarray*}
\end{enumerate}
 \end{theorem}

 \subsubsection{Proof  of   Theorem~\ref{main2-x}}

 Let
$\mathcal{V}=\{V_t\}_{t\in\mathbb{R}}$,
$\mathcal{U}=\{U_t\}_{t\in\mathbb{R}}$,
$\mathcal{V}'=\{V'_t\}_{t\in\mathbb{R}}$  and
  $\mathcal{U}'=\{U_t\}_{t\in\mathbb{R}}$
 be  persistent  modules over $R$.  For any $t\leq s$,  suppose
$ \nu_t^s: V_t\longrightarrow V_s$,
$ \mu_t^s: U_t\longrightarrow U_s$,
 $ {\nu'}_t^s: V'_t\longrightarrow V'_s$,   and
 ${\mu'}_t^s: U'_t\longrightarrow U'_s$
 are the canonical homomorphisms   that define   the  persistent  modules $\mathcal{V}$, $\mathcal{U}$, $\mathcal{V}'$ and $\mathcal{U}'$ respectively.
 Let $\Phi: \mathcal{V}\longrightarrow \mathcal{U}$ and $\Phi': \mathcal{V}'\longrightarrow \mathcal{U}'$ be homomorphisms of  persistent  modules   (cf. \cite[Section~1.3]{pmd}).    We have  a  persistent  sub-$R$-module  $\text{Ker}(\Phi)=\{\text{Ker}(\Phi_t)\}_{t\in\mathbb{R}}$   of  $\mathcal{V}$  and    a  persistent  sub-$R$-module  $\text{Im}(\Phi)=\{\text{Im}(\Phi_t)\}_{t\in\mathbb{R}}$    of     $\mathcal{U}$.
We  also  have a   persistent  quotient $R$-module  $\text{Coker}(\Phi)=\{\text{Coker}(\Phi_t)\}_{t\in\mathbb{R}}$
   of  $\mathcal{U}$.   As  persistent  $R$-modules,
 $\mathcal{V}/ \text{Ker}(\Phi) \cong \text{Im}(\Phi)$  and   $\mathcal{U}/\text{Im}(\Phi)= \text{Coker}(\Phi)$.
  Let $\epsilon>0$.
  \begin{definition}\label{def-998}
   We say that $\Phi$ and $\Phi'$ are {\it strongly $\epsilon$-interleaved}   if all of the followings are satisfied:
\begin{enumerate}[(a).]
\item
there exist two families of homomorphisms $\{\alpha_t:  V_t\longrightarrow V'_{t+\epsilon}\}_{t\in \mathbb{R}}$ and  $\{\alpha'_t:  V'_t\longrightarrow V_{t+\epsilon}\}_{t\in \mathbb{R}}$ such that $\mathcal{V}$ and $\mathcal{V}'$ are strongly $\epsilon$-interleaved via these two families of homomorphisms;

\item
there exist two families of homomorphisms $\{\beta_t:  U_t\longrightarrow U'_{t+\epsilon}\}_{t\in \mathbb{R}}$ and  $\{\beta'_t:  U'_t\longrightarrow U_{t+\epsilon}\}_{t\in \mathbb{R}}$ such that $\mathcal{U}$ and $\mathcal{U}'$ are strongly $\epsilon$-interleaved via these two families of homomorphisms;

\item for any $t\leq s$, the following four diagrams commute
{\small\begin{eqnarray*}
\xymatrix{
V_t \ar[r]^{\nu_t^s} \ar[d]_{\Phi_t} &V_s\ar[d]^{\Phi_s}&V'_t \ar[r]^{{\nu'}_t^s} \ar[d]_{\Phi'_t} &V'_s\ar[d]^{\Phi'_s}&V_t \ar[r]^{\alpha_t} \ar[d]_{\Phi_t} &V'_{t+\epsilon}\ar[d]^{\Phi'_{t+\epsilon}}&V'_t \ar[r]^{\alpha'_t} \ar[d]_{\Phi'_t} &V_{t+\epsilon}\ar[d]^{\Phi_{t+\epsilon}}\\
U_t \ar[r]^{\mu_t^s}   &U_s, & U'_t \ar[r]^{{\mu'}_t^s}   &U'_s, &U_t \ar[r]^{\beta_t}   &U'_{t+\epsilon}, & U'_t \ar[r]^{\beta'_t}   &U_{t+\epsilon}.
}
\end{eqnarray*}}
\end{enumerate}
\end{definition}

\begin{lemma}\label{le-abc}
Suppose two homomorphisms  $\Phi: \mathcal{V}\longrightarrow \mathcal{U}$ and $\Phi': \mathcal{V}'\longrightarrow \mathcal{U}'$ are strongly $\epsilon$-interleaved.  Then the following homomorphisms are also strongly  $\epsilon$-interleaved:
\begin{enumerate}[(i).]
\item
$\Phi:{\rm Ker}(\Phi)\longrightarrow 0$ and $\Phi': {\rm  Ker}(\Phi')\longrightarrow 0$;
\item
$\Phi/{\rm  Ker}: \mathcal{V}/{\rm Ker}(\Phi)\longrightarrow {\rm Im}(\Phi)$  and $\Phi'/{\rm Ker}: \mathcal{V}'/{\rm Ker}(\Phi')\longrightarrow {\rm Im}(\Phi')$;
\item
$\Phi/{\rm  Im}: 0\longrightarrow {\rm Coker}(\Phi)$  and $\Phi'/{\rm Im}: 0\longrightarrow {\rm  Coker}(\Phi')$.
\end{enumerate}
\end{lemma}

 \begin{proof}
Let  $t\in \mathbb{R}$.  Both $\beta_t$ and $\beta'_t$ send $0$ to $0$.  By the  third and the forth  commutative diagrams  in  Definition~\ref{def-998}~(c)  respectively,    $\alpha_t$ sends $\text{Ker}(\Phi_t)$ to $\text{Ker}(\Phi'_{t+\epsilon})$ and  $\alpha'_t$ sends $\text{Ker}(\Phi'_t)$ to $\text{Ker}(\Phi_{t+\epsilon})$.
 Hence   $\Phi:{\rm Ker}(\Phi)\longrightarrow 0$ and $\Phi': {\rm  Ker}(\Phi')\longrightarrow 0$,   which  are the restrictions of $\Phi$ and $\Phi'$ to ${\rm Ker}(\Phi)$ and ${\rm Ker}(\Phi')$ respectively,     are  well-defined.  By the strong $\epsilon$-interleaving property of $\Phi$ and $\Phi'$,    $\Phi$ and $\Phi'$  satisfy (a), (b) and (c) in Definition~\ref{def-998}.   Hence   $\Phi:{\rm Ker}(\Phi)\longrightarrow 0$ and $\Phi': {\rm  Ker}(\Phi')\longrightarrow 0$ satisfy (a), (b) and (c) in Definition~\ref{def-998} as well.   Thus (i) follows.

The following two families
 \begin{eqnarray}
& \{\alpha_t/\text{Ker}:  V_t/\text{Ker}(\Phi_t)\longrightarrow V'_{t+\epsilon}/\text{Ker}(\Phi'_{t+\epsilon})\}_{t\in \mathbb{R}},  \label{eq-xxx9}\\
& \{\alpha'_t/\text{Ker}:  V'_t/\text{Ker}(\Phi'_t)\longrightarrow V_{t+\epsilon}/\text{Ker}(\Phi_{t+\epsilon})\}_{t\in \mathbb{R}}  \label{eq-xxx8}
 \end{eqnarray}
 of homomorphisms   are well-defined.  Via the  two families   (\ref{eq-xxx9})    and   (\ref{eq-xxx8})  of homomorphisms,
the persistent modules $\mathcal{V}/{\rm Ker}(\Phi)$ and $\mathcal{V}'/{\rm Ker}(\Phi')$ are strongly $\epsilon$-interleaved.  By
substituting      $\mathcal{V}$  and $\mathcal{V}'$  in  Definition~\ref{def-998} with       $\mathcal{V}/{\rm Ker}(\Phi)$ and  $\mathcal{V}'/{\rm Ker}(\Phi')$  respectively  and
substituting   $\{\alpha_t\}_{t\in\mathbb{R}}$ and  $\{\alpha'_t\}_{t\in\mathbb{R}}$        in  Definition~\ref{def-998} with       (\ref{eq-xxx9})  and   (\ref{eq-xxx8})      respectively,
the  four   diagrams in (c)  still commute.   This implies that  $\Phi/\text{Ker}$ and $\Phi'/\text{Ker}$ are strongly $\epsilon$-interleaved via  the persistent modules $\mathcal{V}/{\rm Ker}(\Phi)$,  $\mathcal{V}'/{\rm Ker}(\Phi')$,  $\mathcal{U}$ and $\mathcal{U}'$  and  the families $\{\alpha_t/\text{Ker}\}_{t\in\mathbb{R}}$  in (\ref{eq-xxx9}),   $\{\alpha'_t/\text{Ker}\}_{t\in\mathbb{R}}$  in      (\ref{eq-xxx8}),   $\{\beta_t\}_{t\in\mathbb{R}}$ and $\{\beta'_t\}_{t\in\mathbb{R}}$  of homomorphisms.  Thus (ii) follows.

Similar with the proof of (i),    the homomorphisms $\Phi/{\rm  Im}: 0\longrightarrow {\rm Coker}(\Phi)$  and $\Phi'/{\rm Im}: 0\longrightarrow {\rm  Coker}(\Phi')$  as  well   as   the    families
\begin{eqnarray}
&\{\beta_t/\text{Im}:  U_t/\text{Im}(\Phi_t)\longrightarrow U'_{t+\epsilon}/\text{Im}(\Phi'_{t+\epsilon})\}_{t\in \mathbb{R}},  \label{eq-zzx1}
\\
&\{\beta'_t/\text{Im}:  U'_t/\text{Im}(\Phi'_t)\longrightarrow U_{t+\epsilon}/\text{Im}(\Phi_{t+\epsilon})\}_{t\in \mathbb{R}}
\label{eq-zzx2}
\end{eqnarray}
 of homomorphisms   are well-defined.
the persistent modules $\mathcal{U}/{\rm  Im}(\Phi)$ and $\mathcal{U}'/{\rm Im}(\Phi')$ are strongly $\epsilon$-interleaved  via    (\ref{eq-zzx1})    and   (\ref{eq-zzx2}).   By
substituting     $\mathcal{U}$  and $\mathcal{U}'$  in  Definition~\ref{def-998} with      $\mathcal{U}/{\rm  Im}(\Phi)$ and  $\mathcal{U}'/{\rm  Im}(\Phi')$  respectively   and
substituting    $\{\beta_t\}_{t\in\mathbb{R}}$ and  $\{\beta'_t\}_{t\in\mathbb{R}}$        in  Definition~\ref{def-998} with       (\ref{eq-zzx1})  and   (\ref{eq-zzx2})      respectively,
the  four   diagrams in (c)   commute.    This implies that   $\Phi/\text{Im}$ and $\Phi'/\text{Im}$ are strongly $\epsilon$-interleaved via  the persistent modules $\mathcal{V}$,  $\mathcal{V}'$,  $\mathcal{U}//{\rm Im}(\Phi)$ and $\mathcal{U}'/{\rm Im}(\Phi')$  and  the families $\{\alpha_t\}_{t\in\mathbb{R}}$, $\{\alpha'_t\}_{t\in\mathbb{R}}$,  $\{\beta_t/{\rm Im}\}_{t\in\mathbb{R}}$  in (\ref{eq-zzx1})  and $\{\beta'_t/{\rm Im}\}_{t\in\mathbb{R}}$   in  (\ref{eq-zzx2})  of homomorphisms.  Thus (iii) follows.
 \end{proof}

The next corollary~(i), (ii) and (iii) follow from Lemma~\ref{le-abc}~(i), (ii) and (iii) respectively.

\begin{corollary}\label{co-zzz}
Suppose two homomorphisms  $\Phi: \mathcal{V}\longrightarrow \mathcal{U}$ and $\Phi': \mathcal{V}'\longrightarrow \mathcal{U}'$ are strongly $\epsilon$-interleaved.  Then the following  persistent  $R$-modules are also strongly  $\epsilon$-interleaved:
\begin{enumerate}[(i).]
\item
${\rm  Ker}(\Phi)$ and ${\rm Ker}(\Phi')$;
\item
$\mathcal{V}/{\rm Ker}(\Phi)$ and $\mathcal{V}'/{\rm Ker}(\Phi')$, or equivalently, ${\rm Im}(\Phi)$ and ${\rm  Im}(\Phi')$;
\item
${\rm Coker}(\Phi)$  and ${\rm  Coker}(\Phi')$.
\qed
\end{enumerate}
\end{corollary}

Let   $R$ be  a field $\mathbb{F}$.   Choose a  persistent  basis
$b(\text{Ker}(\Phi))
$ for $\text{Ker}(\Phi)$.  By  \cite[Theorem~1.1]{decomp},   we may extend $b(\text{Ker}(\Phi))$  to a  persistent  basis $b(\text{Ker}(\Phi))\sqcup b(\Phi)$ of $\mathcal{V}$,  where $b(\Phi)$ is a  persistent  basis  for $\mathcal{V}/\text{Ker}(\Phi)$.  The  persistent  linear map $\Phi$ sends $b(\Phi)$ bijectively to a  persistent  basis $\Phi(b(\Phi))$ of  $\text{Im}(\Phi)$.  By  \cite[Theorem~1.1]{decomp},  we may extend $\Phi(b(\Phi))$ to a  persistent  basis  $\Phi(b(\Phi))\sqcup b(\text{Coker}(\Phi))$ of $\mathcal{U}$,  where $b(\text{Coker}(\Phi))$ is a  persistent  basis for $\text{Coker}(\Phi)$.
 Similarly,  we have  a persistent  basis $b(\text{Ker}(\Phi'))$ for $\text{Ker}(\Phi')$,  a  persistent basis  $b(\Phi')$ for $\mathcal{V}'/\text{Ker}(\Phi')$,  a  persistent basis $\Phi'(b(\Phi'))$ for $\text{Im}(\Phi')$, and a  persistent basis $b(\text{Coker}(\Phi'))$ for $\text{Coker}(\Phi')$.  By taking the birth-times and the death-times of the elements in the  persistent  bases, we have the corresponding  persistent  diagrams.

 \begin{lemma}\label{le-789}
Let $\mathbb{F}$ be a field.  Suppose the  persistent  $\mathbb{F}$-linear maps $\Phi: \mathcal{V}\longrightarrow\mathcal{U}$ and $\Phi': \mathcal{V}'\longrightarrow \mathcal{U}'$ are strongly $\epsilon$-interleaved.  Then
\begin{enumerate}[(i).]
\item
$d_B^\infty(D({\rm Ker}(\Phi)), D({\rm  Ker}(\Phi')))\leq \epsilon$;
\item
  $d_B^\infty(D({\rm Im}(\Phi)),  D({\rm  Im}(\Phi')))\leq \epsilon$;

\item
$d_B^\infty(D({\rm  Coker}(\Phi)), D({\rm  Coker}(\Phi')))\leq \epsilon$.
\end{enumerate}
 \end{lemma}
\begin{proof}
By applying \cite[Theorem~4.4]{pd2} to Corollary~\ref{co-zzz}~(i), (ii) and (iii) respectively,  we obtain (i), (ii) and (iii) in Lemma~\ref{le-789}.
\end{proof}

 Define the  persistent  diagram of a  persistent  linear map $\Phi$  as the triple
 \begin{eqnarray*}
 D(\Phi)=(D(\text{Ker}(\Phi)), D(\text{Im}(\Phi)),  D(\text{Coker}(\Phi))).
 \end{eqnarray*}
 Define the $L^\infty$-bottleneck distance between two  persistent  linear maps $\Phi$ and $\Phi'$ as   \begin{eqnarray*}
 d_B^\infty(D(\Phi), D(\Phi'))&=& \max\{d_B^\infty(D(\text{Ker}(\Phi)), D(\text{Ker}(\Phi'))),  \\ &&
 d_B^\infty(D(\text{Im}(\Phi)),  D(\text{Im}(\Phi'))),\\ &&
  d_B^\infty(D(\text{Coker}(\Phi)), D(\text{Coker}(\Phi'))) \}.
 \end{eqnarray*}
 The next  corollary  follows from Lemma~\ref{le-789} directly.
 \begin{corollary}\label{pr-bn}
 Suppose two  persistent  linear maps $\Phi$ and $\Phi'$ are strongly $\epsilon$-interleaved.  Then
 \begin{eqnarray*}
  d_B^\infty(D(\Phi), D(\Phi'))\leq \epsilon.     \qed
  \end{eqnarray*}
 \end{corollary}

 \begin{proof}[Proof of  Theorem~\ref{main2-x}]
(i).   Let $\Phi_f$ and $\Phi_g$ be the  persistent  linear maps.   Suppose $||f-g||_\infty\leq \epsilon$.      Then  Definition~\ref{def-998} (a),  (b) and   (c)  hold.    Thus   $\Phi_f$ and $\Phi_g$ are strongly $\epsilon$-interleaved.  By Corollary~\ref{pr-bn}, we have (i).

(ii). Let $\Phi'_{f'}$ and $\Phi'_{g'}$ be the  persistent  linear maps.    Suppose $||f'-g'||_\infty\leq \epsilon$.  Similar with (i),    $\Phi'_{f'}$ and $\Phi'_{g'}$ are strongly $\epsilon$-interleaved.  By Corollary~\ref{pr-bn}, we have (ii).
 \end{proof}

\subsection{The  stability of  the  Mayer-Vietoris  sequences  }

Let  $\mathcal{H}$   and  $\mathcal{H}'$  be  two  hypergraphs.   Let  $\{\mathcal{H}_t\}_{t\in\mathbb{R}}$   be  a
 filtration  of  $\mathcal{H}$  and  let   $\{\mathcal{H}'_t\}_{t\in\mathbb{R}}$   be  a
 filtration  of  $\mathcal{H}'$   such  that   for  each  $t\in \mathbb{R}$,   $\mathcal{H}_t$  and  $\mathcal{H}'_t$  satisfy the  hypothesis  (P)  in  Proposition~\ref{pr-3.8897}.    Then  we  have  the  following  commutative  diagram  of  persistent  homology
{\small\begin{eqnarray*}
\xymatrix{
 \cdots\ar[r]^-{\gamma_{\delta,n+1}}
 & H_n(\delta(\mathcal{H}_t)\cap\delta(\mathcal{H}'_t))\ar[r]^-{\alpha_{\delta,n}}\ar[d]
 & H_n(\delta(\mathcal{H}_t))\oplus  H_n(\delta(\mathcal{H}'_t)) \ar[r]^-{\beta_{\delta,n}}\ar[d]
 &
&
& \\
 \cdots\ar[r]^-{\gamma_{n+1}}
 & H_n(\mathcal{H}_t\cap\mathcal{H}'_t)\ar[r]^-{\alpha_{n}}\ar[d]
 & H_n(\mathcal{H}_t)\oplus  H_n(\mathcal{H}'_t) \ar[r]^-{\beta_{n}}\ar[d]
 &
&
&\\
\cdots\ar[r]^-{\gamma_{\Delta,n+1}}
 & H_n(\Delta(\mathcal{H}_t) \cap\Delta(\mathcal{H}' _t))\ar[r]^-{\alpha_{\Delta,n}}
 & H_n(\Delta(\mathcal{H}_t) )\oplus  H_n(\Delta(\mathcal{H}'_t) ) \ar[r]^-{\beta_{\Delta,n}}
 &
&
&
}\\
~~~\\
\xymatrix{
  &
 & \ar[r]^-{\beta_{\delta,n}}
 &  H_n(\delta(\mathcal{H}_t)\cup\delta(\mathcal{H}'_t)) \ar[r]^-{\gamma_{\delta,n}}\ar[d]
&   H_{n-1}(\delta(\mathcal{H}_t)\cap\delta(\mathcal{H}'_t))\ar[r]^-{\alpha_{\delta,n-1}}\ar[d]
& \cdots\\
 &
 & \ar[r]^-{\beta_{n}}
 &  H_n({\mathcal{H}_t}\cup\mathcal{H}'_t) \ar[r]^-{\gamma_{n}}\ar[d]
&   H_{n-1}(\mathcal{H}_t\cap\mathcal{H}'_t)\ar[r]^-{\alpha_{n-1}}\ar[d]
& \cdots\\
 &
 &  \ar[r]^-{\beta_{\Delta,n}}
 &  H_n(\Delta(\mathcal{H}_t) \cup\Delta(\mathcal{H}' _t)) \ar[r]^-{\gamma_{\Delta,n}}
&   H_{n-1}(\Delta(\mathcal{H}_t) \cap\Delta(\mathcal{H}'_t)  )\ar[r]^-{\alpha_{\Delta,n-1}}
& \cdots
}
\end{eqnarray*}}where  each  row  is  a  long  exact  sequence of  persistent  modules   and each vertical map is a homomorphism   of  persistent  modules   induced by the canonical inclusions  of  hypergraphs.
 We  have
 \begin{eqnarray*}
 &{\rm   Ker} (\alpha_{n})= {\rm  Im} (\gamma_{n+1}),~~~
 {\rm   Ker} (\beta_{n})= {\rm  Im} (\alpha_{n}), ~~~
 {\rm   Ker} (\gamma_{n})= {\rm  Im} (\beta_{n}),\\
 &{\rm   Coker} (\alpha_{n})\cong  {\rm  Im} (\beta_{n}),~~~
  {\rm   Coker} (\beta_{n})\cong  {\rm  Im} (\gamma_{n}),~~~
   {\rm   Coker} (\gamma_{n})\cong  {\rm  Im} (\alpha_{n-1}).
 \end{eqnarray*}
 The  persistent  diagram   of   the   persistent  linear  map   $\alpha_n$  is
 \begin{eqnarray}
 D(\alpha_n)
 &=&(D({\rm  Ker}(\alpha_n)),  D({\rm  Im}(\alpha_n)),  D({\rm   Coker}(\alpha_n)))\nonumber\\
 &=&(D({\rm   Coker}(\beta_{n+1})),  D({\rm   Ker}(\beta_n)),  D({\rm   Im}(\beta_n)))\nonumber\\
 &=&(D({\rm   Im}(\gamma_{n+1})),  D({\rm   Coker}(\gamma_{n+1})),  D({\rm   Ker}(\gamma_n))).
 \label{eq-alpha}
 \end{eqnarray}
 Similarly,  we  can give the explicit expressions for  $ D(\beta_n)$  and $ D(\gamma_n)$.
     Let  $f,g:  \mathcal{H}\longrightarrow   \mathbb{R}$  and   $f',g':  \mathcal{H}'\longrightarrow   \mathbb{R}$   such  that  for each  $t\in\mathbb{R}$,
both  the  pair  $(\mathcal{H}^f_t,{\mathcal{H}'}^{f'}_t)$   and  the  pair  $(\mathcal{H}^g_t,{\mathcal{H}'}^{g'}_t)$
 satisfy the  hypothesis  (P)  in  Proposition~\ref{pr-3.8897}.
 Denote   the   persistent  linear  maps  $\alpha_n$,  $\beta_n$  and  $\gamma_n$  induced  by  by $f$  and  $g$  as  $\alpha_{n,f,g}$,  $\beta_{n,f,g}$   and  $\gamma_{n,f,g}$     respectively.       The  maximum    among  the  three  distances
 \begin{eqnarray*}
 d_B^\infty  (\alpha_{n,f,g}, \alpha_{\delta,n,f',g'}), ~~~
    d_B^\infty  (\beta_{n,f,g}, \beta_{\delta,n,f',g'}), ~~~
     d_B^\infty  (\gamma_{n,f,g}, \gamma_{n,f',g'})
\end{eqnarray*}
equals to the  maximum    among   the  five  distances
\begin{eqnarray*}
 &&d_B^\infty ({\rm {Ker}}(\alpha_{n,f,g}), {\rm {Ker}}(\alpha_{n,f',g'})),  ~~~
  d_B^\infty ({\rm {Ker}}(\alpha_{n-1,f,g}), {\rm {Ker}}(\alpha_{n-1,f',g'})),  \\
&& d_B^\infty ({\rm {Im}}(\alpha_{n,f,g}), {\rm {Im}}(\alpha_{n,f',g'})), ~~~
 d_B^\infty ({\rm {Im}}(\alpha_{n-1,f,g}), {\rm {Im}}(\alpha_{n-1,f',g'})),  \\
 &&  d_B^\infty ({\rm {Coker}}(\alpha_{n,f,g}), {\rm {Coker}}(\alpha_{n,f',g'})),
        \end{eqnarray*}
        which  equals to the  maximum  among  the  five  distances
        \begin{eqnarray*}
 &&d_B^\infty ({\rm {Ker}}(\beta_{n,f,g}), {\rm {Ker}}(\beta_{n,f',g'})),  ~~~
  d_B^\infty ({\rm {Ker}}(\beta_{n-1,f,g}), {\rm {Ker}}(\beta_{n-1,f',g'})),  \\
&& d_B^\infty ({\rm {Coker}}(\beta_{n,f,g}), {\rm {Coker}}(\beta_{n,f',g'})), ~~~
 d_B^\infty ({\rm {Coker}}(\beta_{n+1,f,g}), {\rm {Coker}}(\beta_{n+1,f',g'})),  \\
 &&  d_B^\infty ({\rm {Im}}(\beta_{n,f,g}), {\rm {Im}}(\beta_{n,f',g'}))
        \end{eqnarray*}
        as  well  as   the  maximum  among  the  five  distances
           \begin{eqnarray*}
 &&d_B^\infty ({\rm {Im}}(\gamma_{n,f,g}), {\rm {Im}}(\gamma_{n,f',g'})),  ~~~
  d_B^\infty ({\rm {Im}}(\gamma_{n+1,f,g}), {\rm {Im}}(\gamma_{n+1,f',g'})),  \\
&&d_B^\infty ({\rm {Coker}}(\gamma_{n,f,g}), {\rm {Coker}}(\gamma_{n,f',g'})),  ~~~
  d_B^\infty ({\rm {Coker}}(\gamma_{n+1,f,g}), {\rm {Coker}}(\gamma_{n+1,f',g'})),  \\
 &&  d_B^\infty ({\rm {Ker}}(\gamma_{n,f,g}), {\rm {Ker}}(\gamma_{n,f',g'})).
        \end{eqnarray*}
Analogous   assertions  hold  for   the  persistent  linear maps  $\alpha_{\delta,*}$,  $\beta_{\delta,*}$  and $\gamma_{\delta,*}$
 as   well  as   the  persistent  linear  maps  $\alpha_{\Delta,*}$,  $\beta_{\Delta,*}$  and $\gamma_{\Delta,*}$.

 Choose  a  persistent  module  in the last  diagram.  Denote the  persistent module  induced  by   $f$  and  $g$  as  $M_{f,g}$  and denote the  persistent module  induced  by   $f'$  and  $g'$  as  $M_{f',g'}$.
  Choose  an  arrow  in the last  diagram.   Denote the arrow  induced  by $f$  and  $g$  as  $\Phi_{f,g}$    and    denote the  arrow   induced  by
$f'$  and  $g'$  as  $\Phi_{f',g'}$.

 \begin{theorem}\label{pr-8yrjtjofw}
 If  $||f-f'||_\infty \leq  \epsilon$  and  $||g-g'||_\infty \leq  \epsilon$,  then    $d_B^\infty(D(M_{f,g}), D(M_{f',g'}))\leq \epsilon$   and   $d_B^\infty(D(\Phi_{f,g}), D(\Phi_{f',g'}))\leq \epsilon$.
 \end{theorem}
 \begin{proof}
 Suppose $||f-f'||_\infty \leq  \epsilon$  and  $||g-g'||_\infty \leq  \epsilon$.   Then    $M_{f,g}$ and $M_{f',g'}$ are strongly $\epsilon$-interleaved.    Thus  $d_B^\infty(D(M_{f,g}), D(M_{f',g'}))\leq \epsilon$.   Moreover,  $\Phi_{f,g}$ and $\Phi_{f',g'}$ are strongly $\epsilon$-interleaved.    Thus   $d_B^\infty(D(\Phi_{f,g}), D(\Phi_{f',g'}))\leq \epsilon$.
 \end{proof}

\subsection{The  stability of      the  K\"unneth-type  formulae}
Let  $\mathcal{H}$   and  $\mathcal{H}'$  be  two  hypergraphs.   Let  $\{\mathcal{H}_t\}_{t\in\mathbb{R}}$   be  a
 filtration  of  $\mathcal{H}$  and  let   $\{\mathcal{H}'_t\}_{t\in\mathbb{R}}$   be  a
 filtration  of  $\mathcal{H}'$.   Let $R$  be a principal ideal domain  with  unit $1$.
  By   Proposition~\ref{pr-3.3.996},  we  have  a  commutative diagram  of  persistent   modules
{\small\begin{eqnarray*}
\xymatrix{
  \bigoplus_{p+q+1=n} H_{p+1}(\delta(\mathcal{H}_t))\otimes H_{q+1}(\delta(\mathcal{H}'_t))\ar[r]\ar[d]
&H_{n+1}(\delta(\mathcal{H}_t) *\delta(\mathcal{H}'_t)) \ar[r]\ar[d]
& &  \\
  \bigoplus_{p+q+1=n} H_{p+1}(\mathcal{H}_t)\otimes H_{q+1}(\mathcal{H}'_t)\ar[r]\ar[d]
&H_{n+1}(\mathcal{H}_t *\mathcal{H}'_t)\ar[r]\ar[d]
 &  & \\
  \bigoplus_{p+q+1=n} H_{p+1}(\Delta(\mathcal{H}_t))\otimes H_{q+1}(\Delta(\mathcal{H}'_t))\ar[r]
&H_{n+1}(\Delta(\mathcal{H}_t) *\Delta(\mathcal{H}'_t))\ar[r]
 &  &
}\\
~~~\\
\xymatrix{
 &
&  \ar[r]
&\bigoplus_{p+q+1=n} {\rm  Tor}_R(H_{p+1}(\delta(\mathcal{H}_t)), H_{q}(\delta(\mathcal{H}'_t))) \ar[d]   \\
 &
& \ar[r]
 &\bigoplus_{p+q+1=n} {\rm  Tor}_R(H_{p+1}(\mathcal{H}_t), H_{q}(\mathcal{H}'_t)) \ar[d]  \\
 &
 & \ar[r]
 &\bigoplus_{p+q+1=n} {\rm  Tor}_R(H_{p+1}(\Delta(\mathcal{H}_t)), H_{q}(\Delta(\mathcal{H}'_t)))
}
\end{eqnarray*}}such that each row  is  a  short  exact sequence  of  persistent  modules  and  each  vertical  map  is  a  homomorphism    of  persistent  modules     induced  by    canonical  inclusions  of  hypergraphs.
Let  $R=\mathbb{F}$  be  a  field.  Then  all the  torsions   in  the  last  diagram   are  trivial.

Let  $D=\{(b_i,d_i) \mid  i=1,\ldots,l\}\cup  \{(t,t)\mid  t\in \mathbb{R}\}$  and     $D'=\{(b'_j,d'_j) \mid j=1,\ldots, m\}\cup  \{(t,t)\mid  t\in \mathbb{R}\}$,   where  $  -\infty  \leq   b_i<d_i\leq  +\infty$   for each $1\leq    i\leq   l$  and   $  -\infty  \leq   b'_j<d'_j\leq  +\infty$  for each  $1\leq   j\leq   m$,      be   two  persistent  diagrams.   We  define  the   product  of  $D$  and  $D'$   as
\begin{eqnarray*}
D    D'=  \{(\max (b_i,b'_j),  \min  (d_i,d'_j))\mid  1\leq  i\leq   l,   1\leq  j\leq  m   \}  \cup  \{(t,t)\mid  t\in \mathbb{R}\}.
\end{eqnarray*}
The  multiplicity of $(\max (b_i,b'_j),  \min  (d_i,d'_j))$ is the  product  of  the  multiplicity of $  (b_i,d_i)$  in   $D$    and  the
multiplicity  of    $(b'_j,d'_j)$  in  $D'$.
We  define   the  sum  of   $D$  and  $D'$  as
 \begin{eqnarray*}
D   +  D'=  \{(b_i,d_i) \mid  1\leq  i \leq   l\} \cup  \{(b'_j,d'_j) \mid   1\leq   j\leq   m\} \cup  \{(t,t)\mid  t\in \mathbb{R}\}.
\end{eqnarray*}
For  any  $(b,d)\in  D+D'$,     if   $(b,d)=(b_i,d_i)=(b'_j,d'_j)$  for  some  $1\leq  i\leq  l$  and  some
 $1\leq  j\leq  m$,  then
the  multiplicity   of  $(b,  d)$     is   the  sum  of  the  multiplicity  of   $(b_i,  d_i)$   in  $D$
  and   the   multiplicity  of    $(b'_j,  d'_j)$   in  $ D'$.

\begin{proposition}
We  have
\begin{eqnarray}
 D(\{H_{n+1}(\delta(\mathcal{H}_t) *\delta(\mathcal{H}'_t))\}_{t\in\mathbb{R}})&=& \sum_{p+q+1=n}  D(\{ H_{p+1}(\delta(\mathcal{H}_t))\}_{t\in\mathbb{R}})  \nonumber  \\
 &&  ~~~~~~~~~~D(\{ H_{q+1}(\delta(\mathcal{H}'_t))\}_{t\in\mathbb{R}}), \label{eq-(1)knth}
\\
 D(\{H_{n+1}(\mathcal{H}_t *\mathcal{H}'_t)\}_{t\in\mathbb{R}})&=& \sum_{p+q+1=n}  D(\{ H_{p+1}(\mathcal{H}_t)\}_{t\in\mathbb{R}})   \nonumber \\
 &&  ~~~~~~~~~~D(\{ H_{q+1}(\mathcal{H}'_t)\}_{t\in\mathbb{R}}), \label{eq-(2)knth}
\\
D(\{H_{n+1}(\Delta(\mathcal{H}_t) *\Delta(\mathcal{H}'_t))\}_{t\in\mathbb{R}})&=& \sum_{p+q+1=n}  D(\{ H_{p+1}(\Delta(\mathcal{H}_t))\}_{t\in\mathbb{R}})    \nonumber\\
 &&  ~~~~~~~~~~D(\{ H_{q+1}(\Delta(\mathcal{H}'_t))\}_{t\in\mathbb{R}}).    \label{eq-(3)knth}
\end{eqnarray}
\end{proposition}

\begin{proof}
Take  the  persistent  diagram  of  each  persistent   module  in  the  last  diagram.   The  equations  (\ref{eq-(1)knth}),  (\ref{eq-(2)knth})  and  (\ref{eq-(3)knth})  follow  from   the  first row,  the  second row  and  the  third  row  in  the  last  diagram  respectively.
\end{proof}

 Let  $f,g:  \mathcal{H}\longrightarrow   \mathbb{R}$  and
  let   $f',g':  \mathcal{H}'\longrightarrow   \mathbb{R}$.  Choose  a  persistent  module  in the last  diagram.  Denote the  persistent module  induced  by   $f$  and  $g$  as  $M_{f,g}$  and denote the  persistent module  induced  by   $f'$  and  $g'$  as  $M_{f',g'}$.
   Choose  an  arrow  in the last  diagram.  Denote the arrow  induced  by $f$  and  $g$  by  $\Phi_{f,g}$    and    denote the  arrow   induced  by
$f'$  and  $g'$  by  $\Phi_{f',g'}$.
 \begin{theorem}
 If  $||f-f'||_\infty \leq  \epsilon$  and  $||g-g'||_\infty \leq  \epsilon$,  then     $d_B^\infty(D(M_{f,g}), D(M_{f',g'}))\leq \epsilon$   and    $d_B^\infty(D(\Phi_{f,g}), D(\Phi_{f',g'}))\leq \epsilon$.
 \end{theorem}

 \begin{proof}
 The  proof  is  an  analog of  Theorem~\ref{pr-8yrjtjofw}.
 \end{proof}

 \section{The  constancy  of  persistent  Betti  numbers  of  hypergraphs}

  Given a filtration  $\{\mathcal{H}_t\}_{t\in\mathbb{R}}$  of  a   hypergraph  $\mathcal{H}$,  consider  the  persistent  embedded homology
 \begin{eqnarray}
 H_n^{t,s}(\{\mathcal{H}_t\}_{t\in\mathbb{R}}) &=& {\rm  Im}  ((\iota_t^s)_*)\nonumber\\
 &=&\frac{{\rm  Ker}(\partial_n(\mathcal{H}_t))}{{\rm  Ker}(\partial_n(\mathcal{H}_t))  \cap  {\rm  Im} (\partial_{n+1}(\mathcal{H}_s))},
 \label{eq-phaaa}
 \end{eqnarray}
 where  $t\leq  s$,   $\partial_n(\mathcal{H}_t)$  is  the restriction  of   the  boundary  map  $\partial_n$   of   $\Delta\mathcal{H}$   to  ${\rm  Inf}_n(\mathcal{H}_t)$,     $\iota_t^s$  is the canonical  inclusion  of  $\mathcal{H}_t$  into  $\mathcal{H}_s$  and  $(\iota_t^s)_*$  is  the  induced  homomorphism  of  the  embedded  homology  groups.    Similarly,  consider the  persistent  homology
 \begin{eqnarray}
 H_n^{t,s}(\{\Delta(\mathcal{H}_t)\}_{t\in\mathbb{R}}) &=&{\rm  Im}  (((\iota_\Delta)_t^s)_*)\nonumber\\
 &=&\dfrac{{\rm  Ker}(\partial_n(\Delta(\mathcal{H}_t)))}{{\rm  Ker}(\partial_n(\Delta(\mathcal{H}_t)))  \cap  {\rm  Im} (\partial_{n+1}(\Delta(\mathcal{H}_s)))},  \label{eq-phbbb}\\
  H_n^{t,s}(\{\delta(\mathcal{H}_t)\}_{t\in\mathbb{R}}) &=&{\rm  Im}  (((\iota_\delta)_t^s)_*)  \nonumber\\
& =&\dfrac{{\rm  Ker}(\partial_n(\delta(\mathcal{H}_t)))}{{\rm  Ker}(\partial_n(\delta(\mathcal{H}_t)))  \cap  {\rm  Im} (\partial_{n+1}(\delta(\mathcal{H}_s)))},  \label{eq-phccc}
 \end{eqnarray}
 where
 $\partial_n(\Delta(\mathcal{H}_t))$   (resp.  $\partial_n(\delta(\mathcal{H}_t))$)  is  the restriction  of   the  boundary  map  $\partial_n$   of   $\Delta\mathcal{H}$   (resp.    $\delta\mathcal{H}$)  to  $C_n(\Delta(\mathcal{H}_t);R)$  (resp.  $C_n(\delta(\mathcal{H}_t);R)$)  and
   $(\iota_\Delta)_t^s$  (resp.  $(\iota_\delta)_t^s$)   is   the  canonical  inclusion  of  $\Delta(\mathcal{H}_t)$   (resp.  $\delta(\mathcal{H}_t)$)  into   $\Delta(\mathcal{H}_s)$ (resp.   $\delta(\mathcal{H}_s)$).      Take the coefficients  of     the  persistent  homology  (\ref{eq-phaaa}),   (\ref{eq-phbbb})  and
      (\ref{eq-phccc})   in a field  $\mathbb{F}$.   The    persistent  Betti  numbers  of   the   filtrations  $\{\mathcal{H}_t\}_{t\in\mathbb{R}}$,   $\{\Delta(\mathcal{H}_t)\}_{t\in\mathbb{R}}$   and  $\{\delta(\mathcal{H}_t)\}_{t\in\mathbb{R}}$
       are  respectively
 \begin{eqnarray*}
 \beta_n^{t,s}  & =&  \dim_\mathbb{F}  H_n^{t,s}(\{\mathcal{H}_t\}_{t\in\mathbb{R}}),  \\
  \beta_{\Delta,n}^{t,s}  & =&  \dim_\mathbb{F}   H_n^{t,s}(\{\Delta(\mathcal{H}_t)\}_{t\in\mathbb{R}}),  \\
  \beta_{\delta,n}^{t,s}  & =&  \dim_\mathbb{F}   H_n^{t,s}(\{\delta(\mathcal{H}_t)\}_{t\in\mathbb{R}}).
 \end{eqnarray*}
    With the help of   \cite[Section~2.3]{hiraoka},
    $ \beta_n^{t,s}$   is  the  sum  of  the  multiplicities   $m_{b,d}$  of  the  points   $(b,d)$     in the  persistent diagram      $D(\{H_n(\mathcal{H}_t)\}_{t\in\mathbb{R}})$  such that $b\leq t$  and    $d>s$,
     $ \beta_{\Delta, n}^{t,s}$   is  the  sum  of  the  multiplicities    $m_{b,d}$  of  the  points   $(b,d)$     in the  persistent diagram      $D(\{H_n(\Delta(\mathcal{H}_t))\}_{t\in\mathbb{R}})$   such that $b\leq t$  and    $d>s$,      and
         $ \beta_{\delta, n}^{t,s}   $   is  the  sum  of  the  multiplicities   $m_{b,d}$  of  the  points   $(b,d)$     in the  persistent diagram      $D(\{H_n(\delta(\mathcal{H}_t))\}_{t\in\mathbb{R}})$   such that $b\leq t$  and    $d>s$.
    \begin{theorem}\label{th-pbnb}
    Let  $\mathcal{H}$  be  a  hypergraph  and  let  $f:  \mathcal{H}\longrightarrow \mathbb{R}$.          For  any   $t\leq  s$,
    \begin{enumerate}[(1).]
    \item
if    $\mathcal{H}_t^f$   and   $\mathcal{H}_s^f$
 have all the vertices as $0$-hyperedges and  have the same   strong  simple-homotopy type,    then
 $\beta_n^{t,s}=  \beta_n(\mathcal{H}_t^f)=\beta_n(\mathcal{H}_s^f)$,  where    $\beta_n$  is  the  Betti  number   of  the embedded homology  group;
  \item
  if  $\mathcal{H}_t^f$   and   $\mathcal{H}_s^f$    have the same  weak simple-homotopy type,      then
  $\beta_{\Delta,n}^{t,s}  = \beta_n(\Delta(\mathcal{H}_t^f))=\beta_n(\Delta(\mathcal{H}_s^f))$;
    \item
    if   $\mathcal{H}_t^f$   and   $\mathcal{H}_s^f$ have the same  weak simple-homotopy type via a sequence of weak simplicial collapses and  weak simplicial expansions   such  that   in each  weak elementary  simplicial collapse or weak elementary  simplicial expansion,  $\eta$  is a free face of $\xi$ and one of (a), (b)  and (c) in Lemma~\ref{pr-2.8a182}  is satisfied,       then
  $\beta_{\delta,n}^{t,s}  = \beta_n(\delta(\mathcal{H}_t^f))=\beta_n(\delta(\mathcal{H}_s^f))$.
    \end{enumerate}
    \end{theorem}

  \begin{proof}
For  any  $t\leq  s$,  we  have  $\mathcal{H}_t^f\subseteq   \mathcal{H}_s^f$,   $1\leq  i\leq  k$.

(1).
By  our assumption,  there  is  a  sequence   $\mathcal{H}_s^f=\mathcal{H}^0\to  \mathcal{H}^1\to \cdots\to \mathcal{H}^n=\mathcal{H}_t^f$  of  hypergraphs  such  that    each  arrow  is  a  strong  elementary  simplicial  collapse   and   for  each  $0\leq  j\leq n$,  all  the  vertices of  $\mathcal{H}^j$ are  $0$-hyperedges.   By  Lemma~\ref{pr-2.3.main},
for  each  $1\leq  j\leq  n$,
the  canonical  inclusion  of  $\mathcal{H}^{j}$  into  $\mathcal{H}^{j-1}$  induces  an  isomorphism  from  $H_*(\mathcal{H}^{j})$  to  $H_*(\mathcal{H}^{j-1})$.    Thus    the  inclusion of $\mathcal{H}_t^f$  into $\mathcal{H}_s^f$  induces an isomorphism  from  $H_*(\mathcal{H}_t^f)$ to   $H_*(\mathcal{H}_s^f)$.    Therefore,   $\beta_n^{t,s}=  \beta_n(\mathcal{H}_t^f)=\beta_n(\mathcal{H}_s^f)$.

(2).  By  our assumption,  there  is  a  sequence   $\mathcal{H}_s^f=\mathcal{H}^0\to  \mathcal{H}^1\to \cdots\to \mathcal{H}^n=\mathcal{H}_t^f$  of  hypergraphs  such  that    each  arrow  is  a  weak  elementary  simplicial  collapse.   By   Lemma~\ref{pr-x218},   for  each  $1\leq  j\leq  n$,
the  canonical  inclusion  of  $\mathcal{H}^{j}$  into  $\mathcal{H}^{j-1}$  induces  an  isomorphism  from  $H_*(\Delta\mathcal{H}^{j})$  to  $H_*(\Delta\mathcal{H}^{j-1})$.   Thus     the  inclusion of $\mathcal{H}_t^f$  into $\mathcal{H}_s^f$  induces an isomorphism  from  $H_*(\Delta(\mathcal{H}_t^f))$ to   $H_*(\Delta(\mathcal{H}_s^f))$.    Therefore,  $\beta_{\Delta,n}^{t,s}  = \beta_n(\Delta(\mathcal{H}_t^f))=\beta_n(\Delta(\mathcal{H}_s^f))$.

(3).   By  our assumption,  there  is  a  sequence   $\mathcal{H}_s^f=\mathcal{H}^0\to  \mathcal{H}^1\to \cdots\to \mathcal{H}^n=\mathcal{H}_t^f$  of  hypergraphs  such  that    each  arrow  is  a  weak  elementary  simplicial  collapse  satisfying  the hypothesis in  (3).
 By   Lemma~\ref{pr-2.8a182},    for  each  $1\leq  j\leq  n$,
the  canonical  inclusion  of  $\mathcal{H}^{j}$  into  $\mathcal{H}^{j-1}$  induces  an  isomorphism  from  $H_*(\delta\mathcal{H}^{j})$  to  $H_*(\delta\mathcal{H}^{j-1})$.   Thus   the  inclusion of $\mathcal{H}_t^f$  into $\mathcal{H}_s^f$  induces an isomorphism     from   $H_*(\delta(\mathcal{H}_t^f))$ to   $H_*(\delta(\mathcal{H}_s^f))$.    Therefore,  $\beta_{\delta,n}^{t,s}  = \beta_n(\delta(\mathcal{H}_t^f))=\beta_n(\delta(\mathcal{H}_s^f))$.
  \end{proof}

{\small

\bigskip

Shiquan Ren

Address:
School  of  Mathematics and Statistics,  Henan University,  Kaifeng   475004,  China.

e-mail:  renshiquan@henu.edu.cn

   \bigskip

Jie Wu  (corresponding  author)

Address: Yanqi Lake Beijing Institute of Mathematical Sciences and Applications, Beijing 101408, China.

  e-mail: wujie@bimsa.cn

}
 \end{document}